\documentclass[11pt,reqno]{amsart}

\usepackage[utf8]{inputenc}
\usepackage{fullpage}
\usepackage{color}
\usepackage{url}

\usepackage[T1]{fontenc}
\usepackage[utf8]{inputenc}
\usepackage{lmodern}
\usepackage[english]{babel}

\usepackage{mathtools}
\usepackage{amssymb}
\usepackage{amsmath}
\usepackage{amsthm}
\usepackage{amssymb}
\usepackage{amsrefs}
\usepackage{mathrsfs}
\usepackage{dsfont}
\usepackage{bbm}
\usepackage{stmaryrd}
\usepackage{eucal}

\usepackage[all]{xy}
\usepackage{tikz-cd}

\usepackage{hyperref}
\definecolor{dark-red}{rgb}{0.5,0.0,0.0}
\definecolor{dark-blue}{rgb}{0.0,0.0,0.4}
\definecolor{dark-green}{rgb}{0,0.4,0}
\hypersetup{
    colorlinks, linkcolor=dark-red,
    citecolor=dark-green, urlcolor=dark-blue
}

\usepackage{enumitem}

\usepackage[colorinlistoftodos]{todonotes}

\newcommand{\R}{\Bbb{R}}
\newcommand{\C}{\Bbb{C}}

\newcommand{\U}{\mathrm{U}}

\newcommand{\bbm}{\mathbbm}
\newcommand{\Q}{\mathbb{Q}}

\newcommand{\E}{\mathcal{E}}

\newcommand{\J}{\mathcal{J}}

\DeclareMathOperator{\colim}{\mathrm{colim}}

\DeclareMathOperator{\hocolim}{\mathrm{hocolim}}
\DeclareMathOperator{\holim}{\mathrm{holim}}
\DeclareMathOperator{\hofibre}{\mathrm{hofibre}}

\newcommand{\bigslant}[2]{{\raisebox{.2em}{$#1$}\left/\raisebox{-.2em}{$#2$}\right.}}

\DeclareMathOperator{\homog}{--homog--}

\DeclareMathOperator{\poly}{--poly--}

\newcommand{\s}{\mathsf{Sp}}
\DeclareMathOperator{\T}{\mathsf{Top}_\ast}

\DeclareMathOperator{\Hom}{\mathrm{Hom}}

\newcommand{\cofrep}{\widehat{c}}
\newcommand{\Ev}{\mathrm{Ev}}
\newcommand{\fibrep}{\widehat{f}}
\DeclareMathOperator{\id}{\mathrm{Id}}
\DeclareMathOperator{\ind}{\mathrm{ind}}
\DeclareMathOperator{\res}{\mathrm{res}}

\DeclareMathOperator{\BU}{\mathrm{BU}}

\renewcommand{\O}{\mathrm{O}}

\newcommand{\CJ}{\J^{\mathbf{R}}}
\newcommand{\CI}{\mathbf{R}}

\newcommand{\Cgamma}{\gamma^\mathbf{R}}
\newcommand{\CE}{C_2 \ltimes \E^\mathbf{R}}
\newcommand{\os}{\mathsf{Sp}^\mathbf{O}}
\newcommand{\us}{\mathsf{Sp}^\mathbf{U}}

\theoremstyle{definition}
\newtheorem{thm}{Theorem}[section]
\newtheorem{prop}[thm]{Proposition}

\newtheorem{lem}[thm]{Lemma}
\newtheorem{cor}[thm]{Corollary}

\newtheorem{ex}[thm]{Example}

\newtheorem{definition}[thm]{Definition}
\newtheorem{rem}[thm]{Remark}

\newtheorem{alphtheorem}{Theorem}[section]

\newtheorem{alphprop}[alphtheorem]{Proposition}

  \makeatletter
\newcommand*{\centerfloat}{%
  \parindent \z@
  \leftskip \z@ \@plus 1fil \@minus \textwidth
  \rightskip\leftskip
  \parfillskip \z@skip}
\makeatother

  \newcommand{\adjunction}[4]{
\xymatrix{
#1:#2 \ar@<0.7ex>[r] &
\ar@<0.7ex>[l] #3:#4
}}

\begin{document}

\title{Unitary Functor Calculus with Reality}
\author{Niall Taggart}
\address{Max Planck Institute for Mathematics, Vivatsgasse 7, 53111 Bonn, Germany}
\email{ntaggart@mpim-bonn.mpg.de}
\thanks{The author wishes to thank David Barnes for numerous helpful and insightful conversations and suggestions on this material.}
\date{\today}
\maketitle

\begin{abstract}
We construct a calculus of functors in the spirit of orthogonal calculus, which is designed to study ``functors with reality'' such as the Real classifying space functor, $\BU_\R(-)$. The calculus produces a Taylor tower, the $n$--th layer of which is classified by a spectrum with an action of $C_2 \ltimes \U(n)$. 

We further give model categorical considerations, producing a zig--zag of Quillen equivalences between spectra with an action of $C_2 \ltimes \U(n)$ and a model structure on the category of input functors which captures the homotopy theory of the $n$--th layer of the Taylor tower. 
\end{abstract}

\setcounter{tocdepth}{1}
{\hypersetup{linkcolor=black} \tableofcontents}

\section*{Introduction}

The orthogonal and unitary calculi \cite{We95,Ta19} systematically study $\J$--spaces where $\J$ is the category of finite--dimensional real inner product spaces or complex inner product spaces, respectively. The key idea behind these calculi is to approximate a given functor by a sequence of polynomial functors, similar to Taylor's series from differential calculus. A natural question to ask is: \emph{what can we say when the $\J$--spaces come with symmetry in the form of a group action?} For Goodwillie calculus, this equivariance has been studied by Dotto \cite{Do16, Do16b, Do17}, and Dotto and Moi \cite{DM16}. An initial step in a much larger project to understand equivariant orthogonal calculus is the following calculus with reality. This is unitary calculus, constructed to take into account the $C_2$--action on the category of complex inner product spaces given by complex conjugation. In \cite{Ta19,Ta20}, we explained the strong analogy of orthogonal and unitary calculi with real and complex topological $K$--theory. In fact, this analogy was the motivation behind the comparisons of \cite{Ta20}. The calculus with reality considered here fits into this analogy, taking the place of $K$--theory with reality, due to Atiyah \cite{At66}, hence the `with reality' appellation.

The idea is as in the orthogonal and unitary calculus, but suitably altered to take into account the precise equivariance which arises from complex conjugation. We define the notion of a polynomial functor with reality, and construct a polynomial approximation functor which is given as the fibrant replacement in a suitable model structure on the category of functors with reality, see Proposition \ref{prop: n--poly model structure}. 

Using localisation techniques we produce a model structure on the category of functors with reality, which captures the homotopy theory of $n$--homogeneous functors, in particular, the $n$-th layer of the Taylor tower, see Proposition \ref{prop: homogeneous model structure with reality}. We thus produce a zig--zag of Quillen equivalences
\[
\xymatrix@C+1cm{
n \homog \CE_0
\ar@<-1 ex>[rr]_(0.6){\ind_0^n \varepsilon^*}
&&
C_2 \ltimes \U(n)\E_n^\mathbf{R}
\ar@<-1ex>[ll]_(0.4){\res_0^n/\U(n)}
\ar[]!<2ex,1ex>;[ddll]!<-3ex,1ex>_{(\xi_n)_!}
 \\
 &&
 \\
\CE_1[\U(n)]
\ar[]!<-2ex,-1ex>;[uurr]!<3ex,-1ex>_{(\xi_n)^*}
\ar@<-1ex>[rr]_{\psi}
&&
\s^\mathbf{O}[C_2 \ltimes \U(n)]
\ar@<-1ex>[ll]_{L_\psi}
}
\]
in Theorems \ref{thm: intermediate cat as E_1}, \ref{thm: E_1 as OS}, and \ref{thm: QE of E_0 and E_n}, which allows for one to characterise $n$--homogeneous functors as (orthogonal) spectra with an action of $C_2 \ltimes \U(n)$. The end result is a Taylor tower in which the $n$--th layer is classified by these spectra with an action of $C_2 \ltimes \U(n)$. 

The situation is noticeably different to the orthogonal and unitary calculus, it is convenient to have an extra step in the zig--zag of Quillen equivalences. The extra step indicates a sensitivity of the calculus to any introduction of extra equivariance. This in turn will lead to more subtle calculations, for example, an enhancement of \cite{Ar02} to the consideration of the Real classifying space of the unitary group functor, $\BU_\R(-)$. 

In detail, in the orthogonal calculus, Barnes and Oman \cite{BO13} constructed a zig--zag of Quillen equivalences using only one intermediate step between their $n$--homogeneous model structure and orthogonal spectra with an action of $\O(n)$. In the unitary calculus, the author \cite{Ta19} first gave a zig--zag between the unitary $n$--homogeneous model structure and unitary spectra with an action of $\U(n)$, and then provided a further Quillen equivalence between unitary spectra with an action of $\U(n)$ and orthogonal spectra with an action of $\U(n)$. This extra step could be composed into the Quillen equivalence between the unitary intermediate category and unitary spectra with an action of $\U(n)$, \cite[Theorem 6.8]{Ta19}, since left (resp. right) Quillen functors compose to give left (resp. right) Quillen functors. In the `with reality' setting, none of the Quillen equivalences may be composed to reduce the length of the zig--zag, as it would require composing left Quillen functors with right Quillen functors, which are neither left nor right Quillen in general.

As an interesting aside, we further strengthen the idea of calculus with reality being in analogy with the $KR$--theory of Atiyah, by giving an equivalence of categories between the category $\CE_1$ which features in our zig-zag and  the Real spectra (see Definition \ref{def: real spectra}) of Schwede, see Proposition \ref{prop: real spectra and E1}. Combining such an equivalence of categories with our zig--zag of Quillen equivalences indicates that the homotopy theory of $n$--homogeneous functors with reality is equivalent to the homotopy theory of Real spectra with an action of $\U(n)$.

\subsection*{Main Results and Organisation}

In Section \ref{section: equivariant stiefel combinatorics}, we establish the Real version of the Stiefel combinatorics of Weiss \cite[Section 1]{We95}. These Real Stiefel combinatorics are crucial to constructing the derivatives of a functor with reality and understanding the homotopy theory of polynomial functors. In particular, we verify the `crucial' result for the existence of the calculus in Proposition \ref{hocolim sphere homeo}. 

\begin{alphprop}
The sphere bundle $S\Cgamma_{n+1}(V,W)$ is $C_2$--homeomorphic to
\[
\underset{0 \neq U \subset \C^{n+1}}{\hocolim}~\CJ(U \oplus V, W).
\]
\end{alphprop}

We define polynomial functors, the polynomial approximation functor and give a description of a model category which captures the homotopy theory of $n$--polynomial functors in Section \ref{section: poly with reality}. This model structure is given as a left Bousfield localisation of the projective model structure on the category of functors with reality. 

We turn our attention to $n$--homogeneous functors in Section \ref{section: homogeneous functors}, the main example being the $n$--th layer of the Taylor tower. These are functors which are $n$-polynomial and have trivial $(n-1)$--polynomial approximation. We further give a model structure on the category of functors with reality which captures the homotopy theory of $n$--homogeneous functors. This model structure is a right Bousfield localisation of the $n$--polynomial model structure. We achieve a Taylor tower of the following form
\[
\xymatrix@C+1cm{
		&			&	\ar@/_1pc/[dl]	F \ar[d]	  \ar@/^1pc/[drr]  \ar@/^1.3pc/[drrr]   &			     	&			& \\
 \cdots \ar[r] & T_{n+1}F \ar[r]_{r_{n+1}} & T_nF \ar[r]_{r_n} & \cdots \ar[r]_{r_2} & T_1F \ar[r]_{r_1} & T_0F \\
 & D_{n+1}F \ar[u] & D_nF \ar[u] & & D_1F \ar[u] &\\
}
\]
with the layers of the Taylor tower being cofibrant--fibrant in the $n$--homogeneous model structure. 

In Section \ref{section: derivative}, we define the derivative of a functor and show that the $(n+1)$--st derivative of an $n$--polynomial functor with reality is trivial. We further construct the intermediate categories, which are the natural home to the derivatives, and produce a stable model structure on the $n$--th level intermediate category.

In Section \ref{section: intermediate category as spectra}, we give two out of the three steps in our zig--zag of Quillen equivalences by demonstrating a zig--zag of Quillen equivalences between the intermediate category and the category of spectra with an action of $C_2 \ltimes \U(n)$. The following results are given as Theorem \ref{thm: intermediate cat as E_1} and Theorem \ref{thm: E_1 as OS}, respectively. 

\begin{alphtheorem}
The adjoint pair
\[
\adjunction{(\xi_n)_!}{C_2 \ltimes \U(n)\E_n^\mathbf{R}}{\CE_1[\U(n)]}{(\xi_n)^*}
\]
is a Quillen equivalence when both categories are equipped with their stable model structures.
\end{alphtheorem}

\begin{alphtheorem}
The adjoint pair
\[
\adjunction{L_{\psi}}{\os[C_2 \ltimes \U(n)]}{\CE_1[\U(n)]}{\psi}
\]
is a Quillen equivalence when both categories are equipped with their stable model structures.
\end{alphtheorem}

We further give -- as Proposition \ref{prop: real spectra and E1} -- an equivalence of categories between the category of Real spectra of Schwede (Definition \ref{def: real spectra}) and $\CE_1$. 

\begin{alphprop}
The category of Real spectra, $C_2 \ltimes \us$, is equivalent to the category $\CE_1$.
\end{alphprop}

In Section \ref{section: diff as QF}, we prove -- as Theorem \ref{thm: QE of E_0 and E_n} -- that the differentiation functor is a right Quillen functor as part of a Quillen equivalence between the $n$--homogeneous model structure on the category of functors with reality, and the $n$--stable model structure on the $n$--th level intermediate category.

\begin{alphtheorem}
For $n\geq 0$, the Quillen adjunction
\[
\adjunction{\res_0^n/\U(n)}{C_2 \ltimes \U(n)\E_n^\mathbf{R}}{n\homog\CE_0}{\ind_0^n \varepsilon^*}
\]
where $C_2 \ltimes \U(n)\E_n^\mathbf{R}$ is equipped with the $n$--stable model structure, is a Quillen equivalence.
\end{alphtheorem}

As Theorem \ref{thm: homog char}, we classify the $n$--homogeneous functors in a similar way to other calculi, \cite{Go03, We95, Ta19}. 

\begin{alphtheorem}
If $F$ is a $n$--homogeneous functor with reality, then $F$ is levelwise weakly equivalent to the functor
\[
V \mapsto \Omega^\infty[(S^{\C^n \otimes V} \wedge \Theta_F^n)_{h\U(n)}]
\]
where $\Theta_F^n$ is a (orthogonal) spectrum with an action of $C_2 \ltimes \U(n)$. 
\end{alphtheorem}

In the final section, Section \ref{section: examples}, we discuss a few examples. These examples are similar to those from unitary calculus. This is expected, and further strengthens the analogy between the calculi and various versions of $K$--theory, in particular that unitary calculus can be obtained from the calculus with reality just as topological $K$--theory can be obtained from $KR$--theory, by `forgetting' the $C_2$--action.

\subsection*{Notation and Conventions}
We fix once and for all an isometric isomorphism $\C^n \cong \C \otimes \R^n$. This isomorphism fixes the complex conjugation on $\C^n$ as the one coming from $\C \otimes \R^n$. We denote by $C_2 \ltimes \U(n)$ the semi-direct product of the group of two elements, and the unitary group of degree $n$. We use $g$ to denote the non-identity element of $C_2$. For a complex vector space $V$, we denote by $nV$ the tensor product $\C^n \otimes_\C V$. 

We denote by $\T$ the category of based compactly generated weak Hausdorff spaces, and always equip $\T$ with the Quillen model structure with sets of generating cofibrations and generated acyclic cofibrations $I$ and $J$ respectively. Given an adjunction, $\adjunction{F}{\mathcal{C}}{\mathcal{D}}{G}$, the left adjoint will always be written on top or to the left had side, and the use of $\sim$ denotes an equivalence of categories.

\section{Real Stiefel Combinatorics}\label{section: equivariant stiefel combinatorics}
A theory of calculus is reliant on the notions of polynomials and derivatives. In the orthogonal and unitary calculus, the derivatives were constructed via relations between particular indexing categories, called the $n$--th jet categories (see \cite[Section 1]{We95} and \cite[Definition 4.3]{Ta19}). The relation between these categories gives the adjunctions used in the Quillen equivalences of \cite{BO13} and \cite{Ta19}. This task is subtly more difficult when dealing with the calculus with reality. To have a well-defined theory we need to carefully choose our indexing category to fix a complex conjugation. We start by fixing such a universe, constructing the $n$--th jet categories, and demonstrating relationships between these categories. 

\subsection{The universe}
Unitary calculus is indexed on the universe $\C^\infty$, i.e. in \cite{Ta19}, we considered functors which take values on finite--dimensional inner product subspaces of $\C^\infty$. In the calculus with reality setting, we want the universe, and all its finite--dimensional subspaces, to be closed under complex conjugation, as without such a closure condition, we would be unable to assign a $C_2\T$--enrichment on the category of input functors. For this, $\C^\infty$ is inappropriate, as is highlighted by the following example. 

\begin{ex}\label{ex: why real basis}
Consider the complex line $\ell \subset \C^2$ spanned by the vector $(1,i)$. This line is not closed under the inherited complex conjugation from $\C^2$, since $\overline{(1,i)} = (1, -i)$ is not proportional to $(1,i)$. Moreover, there are many choices for a complex conjugation on the line spanned by $(1,i)$. If we consider the above complex conjugation, the set of fixed points is clearly the trivial inner product space $\{\underline{0}\}$. However,  this line is isomorphic to $\C$, via the isomorphism $\varphi : \ell \longrightarrow \C, (1,i) \mapsto 1$. This isomorphism defines a complex conjugation on the line $\ell$ via $\overline{(1,i)} = \varphi^{-1}(\overline{\varphi(1,i)})$. This fixed point set of this complex conjugation is the real line. These complex conjugations do not agree, and there are several choices of the isomorphisms between $\ell$ and $\C$, and hence serval different choices for the complex conjugation.
\end{ex}

We choose the universe as $\C \otimes \R^\infty$. Within this universe we consider inner product spaces of the form $\C \otimes V$, for $V \subset \R^\infty$. Complex conjugation is then given by $\overline{c \otimes v} = \overline{c} \otimes v$ for $c \in \C, v \in V$, and the standard complex conjugation on $\C$.

\begin{rem}
The consideration of inner product spaces of the form $\C \otimes V$ is equivalent to requiring complex inner product spaces which have a real basis, that is, finite--dimensional complex inner product spaces such that there exists a basis $\beta$ consisting of only real inner products. For example, the standard basis on $\C^n$ is real. It is not hard to show that $V = \C \otimes V'$ if and only if $V'$ has a real basis. Assume $V$ is a complex finite--dimensional inner product space with real basis $\beta =\{\beta_1, \dots, \beta_k\}$. Then complex conjugation is given by
\[
\overline{v} = \overline{\sum_{i=1}^k \lambda_k \beta_k} = \sum_{i=1}^k \overline{\lambda_k} \beta_k
\]
where $\sum_{i=1}^k \lambda_k \beta_k$ is the unique expression of $v$ in terms of the basis vectors.
\end{rem}

\subsection{The indexing categories}
With the correct universe in place, the construction of the indexing categories is formal, and follows the orthogonal and unitary calculus versions, \cite{We95, Ta19}. 

\begin{definition}
Let $\CJ$ be the category of finite--dimensional real-based complex inner product subspaces of $\C \otimes \R^\infty$ with complex linear isometries. Define $\CJ_0$ to be the category with the same objects as $\CJ$ and morphisms $\CJ_0(U,V) = \CJ(U,V)_+$.
\end{definition}

There categories are $C_2\T$--enriched; they are topologised as the Stiefel manifold of $\dim(U)$--frames in $V$, with $C_2$ acting on the morphism spaces by conjugation by complex conjugation, i.e. for $f \in \CJ(U,V)$, the non--trivial element $g \in C_2$ and $u \in U$,
\[
(g \cdot f)(u) = gf(gu) = \overline{f(\overline{u})}.
\]

\begin{rem}
This enrichment is the underlying reason for considering inner product spaces with a real basis. As an example, the space of linear isometries from the line $\ell$ of Example \ref{ex: why real basis} to itself does not have a well defined $C_2\T$--enrichment.
\end{rem}

The $n$--th jet categories are also constructed similarly to orthogonal and unitary calculus. Sitting over the space of linear isometries $\CJ(U,V)$ is the $n$--th complement vector bundle
\[
\Cgamma_n(U,V) = \{ (f,x) \ : \ f \in \CJ(U,V), x \in \C^n \otimes f(U)^\perp\}
\]
where we have identified the cokernel of $f$ with the orthogonal complement of $f(U)$ in $V$. This vector bundle comes with a $C_2$--action induced from the diagonal $C_2$--action on $\CJ(U,V) \times( \C^n \otimes V)$.

\begin{definition}
Define the $n$--th jet category $\CJ_n$ to be category with the same objects as $\CJ$ and morphism space $\CJ_n(U,V)$, the Thom space of the vector bundle $\Cgamma_n(U,V)$.
\end{definition}

The spaces $\CJ_n(U,V)$ inherit a $C_2$--action from the vector bundle $\Cgamma_n(U,V)$, hence the $n$--th jet categories are $C_2\T$--enriched. As with orthogonal calculus \cite[Section 1]{We95} there are important relations between the morphism spaces of the $n$--th jet categories for varying $n$. These relations are crucial when considering the relationships between polynomial functors and derivatives. The following is the `with reality' version of \cite[Proposition 1.2]{We95}. To ease notation, given $V \in \CJ$, we denote the tensor product $\C^n \otimes V$ by $nV$. 

\begin{prop}\label{prop: cofibre sequence of jet categories}
The reduced mapping cone of the restricted composition
\[
\CJ_n(\C \oplus U,V) \wedge S^{2n} \longrightarrow \CJ_n(U,V)
\]
is $C_2$--homeomorphic to $\CJ_{n+1}(U,V)$, where we have identified $S^{2n}$ $C_2$-equivariantly with the closure of the subspace $\CJ_n(U, \C \oplus U)$ of pairs $(i, x)$ with $i : V \longrightarrow \C \oplus V$ the standard inclusion.
\end{prop}
\begin{proof}
Denoting by $\mathcal{P}$ the relevant mapping cone, which is a quotient of
\[
[0, \infty] \times \CJ_n(U,V) \times S^{2n}
\]
the isomorphism is specified by
\[
\mathcal{P} \longrightarrow \CJ_{n+1}(U,V),\quad (t,f,y,z) \mapsto (f|_U, y + (f|_\C)(z) + \nu((f|_\C)(1)))
\]
where $\nu: V \longrightarrow (n+1)V$ $C_2$-equivariantly identifies $V$ with the orthogonal complement of $nV$ in $(n+1)V$. The group $C_2$ acts on $(t, f, y, z)$ as $(t, gfg, gy, gz)$. Under this isomorphism, this is mapped to
\[
((gfg)|_U, gy + ((gfg)|_\C)(gz) + t\nu(((gfg)|_\C)(1)).
\]
This can be rewritten as
\[
(g \cdot (f|_U), gy + g((f|_\C)(z)) + g(\nu((f|_\C)(1))).
\]
which is precisely the image of the $C_2$--action on the image of $(t,f,y,z)$ under the isomorphism.
\end{proof}


Another essential result for both orthogonal and unitary calculi is the ability to write the sphere bundle of the $n$--th complement vector bundle as a homotopy colimit (see \cite[Proposition 4.2]{We95} and \cite[Theorem 4.1]{Ta19}). The same result holds in this context.

\begin{prop}\label{hocolim sphere homeo}
The sphere bundle $S\Cgamma_{n+1}(V,W)$ is $C_2$--homeomorphic to
\[
\underset{0 \neq U \subset \C^{n+1}}{\hocolim}~\CJ(U \oplus V, W).
\]
\end{prop}
\begin{proof}
The unitary version is given in \cite[Theorem 4.1]{Ta19}. It is enough to check that this construction is suitably $C_2$--equivariant. The homeomorphism is constructed as follows.
\[
\Psi : \CJ_{n+1}(V,W)\setminus \CJ_0(V,W) \longrightarrow (0,\infty) \times \hocolim_U \CJ(U \oplus V,W),
\]
which is given by $\Psi(f,x) = (t,G,z,p)$, where
\begin{enumerate}
\item $G: [k] \longrightarrow \mathbf{R}_{\leq n+1}$ is a functor given by
\[
r \mapsto E(\lambda_0) \oplus \cdots \oplus E(\lambda_{k-r})
\]
where the $E(\lambda_i)$ is the eigenspaces which constitute direct summands of $\C^{n+1}$ corresponding to the distinct eigenvalues of $x^\ast x$, where we have identified $\C^{n+1} \otimes f(V)^\perp \cong \Hom(\C^{n+1}, f(V)^\perp)$.
\item $z \in \J(G(0) \oplus V, W)$ a linear isometry, given by
\[
z =
\begin{cases}
f \ \text{on} \ V \\
\lambda_i^{-1/2}x \ \text{on} \ E(\lambda_i).
\end{cases}
\]
\item $p \in \Delta^k$ given by the barycentric coordinates
\[
\lambda_k^{-1}(\lambda_0, \lambda_1 - \lambda_0, \cdots, \lambda_k - \lambda_{k-1}).
\]
\item $t = \lambda_k \leq 0$.
\end{enumerate}
We check that $\Psi$ is $C_2$--equivariant. Let $(f,x) \in \J_{n+1}^\mathbf{R}(V,W)\setminus \J_0^\mathbf{R}(V,W)$. Then $g\cdot(f,x) = (gfg, g\cdot x)$. Following $g\cdot x \in \C^{n+1} \otimes f(V)^\perp$ through the isomorphism $\Hom(\C^{n+1}, f(V)^\perp)$, we see that when thinking of the vector $g\cdot x$ as a map, it is equal to the map
\[
gxg : \C^{n+1} \longrightarrow f(V)^\perp, (c_i) \mapsto \overline{x(\overline{(c_i)})}.
\]
The map $g\cdot x$ had an adjoint $(g\cdot x)^* : f(V)^\perp \longrightarrow \C^{n+1}$, and hence we get a self--adjoint map
\[
(g\cdot x)^*(g\cdot x) : \C^{n+1} \longrightarrow \C^{n+1}.
\]
Note that complex conjugation defines a self map $g: \C^{n+1} \longrightarrow \C^{n+1}$ which is an isometric isomorphism, and hence the adjoint of $g$ equals the inverse of $g$. It follows that $(g\cdot x)^* = g(x^*)g$, and hence
\[
(g\cdot x)^*(g\cdot x) : \C^{n+1} \longrightarrow \C^{n+1}, (c_i) \mapsto \overline{x^*(x(\overline{(c_i)}))}.
\]
The eigenvectors of $(g\cdot x)^*(g\cdot x)$ are hence the complex conjugate of the eigenvectors of $x^*x$, with the same corresponding eigenvalues, in particular, these eigenvalues are distinct, positive and real (see \cite[Lemma pg.329 and Theorem 6.24]{FIS89}). Denote by $\overline{E(\lambda_i)}$ the eigenspace of eigenvectors of $(g\cdot x)^*(g\cdot x)$ associated to the eigenvalue $\lambda_i$, i.e. $\overline{E(\lambda_i)}$ is the space of complex conjugates of the eigenvectors of $x^*x$ associated with the eigenvalue $\lambda_i$.

The image of $(gfg, g\cdot x)$ under $\Phi$ is $(t, gG, gz, p)$, where
\begin{description}
\item[1)] the functor $gG : [k] \longrightarrow \mathbf{R}_{\leq n+1}$ is given by
\begin{equation*}
r \mapsto \overline{E(\lambda_0)} \oplus \dots \oplus \overline{E(\lambda_{k-r})}
\end{equation*}
\item[2)] the linear isometry $z \in \J^\mathbf{R}(gG(0) \oplus V, W)$ is given by
\begin{equation*}
z = \begin{cases}
gfg \ \text{on} \ V \\
\lambda_i^{-1/2} (g\cdot x) \ \text{on} \ \overline{E(\lambda_i)}
\end{cases}
\end{equation*}
\item[3)] the point $p \in \Delta^k$ is given by the barycentric coordinates
\begin{equation*}
\lambda_k^{-1} \cdot (\lambda_0, \lambda_1 - \lambda_0, \lambda_2-\lambda_1, \dots, \lambda_k - \lambda_{k-1})
\end{equation*}
\item[4)] and $t=\lambda_k >0$.
\end{description}
This description matches the description of $g\cdot \Phi(f,x)$, and hence $\Phi$ is $C_2$--equivariant.
\end{proof}

\begin{rem}
The above result is often described by the experts as the `crucial' result for the (orthogonal) calculi to work. It is instrumental in constructing the $n$--polynomial model structure. If one were to consider a `genuine' equivariant version of orthogonal calculus, that is studying $G\T$--enriched functors from $\J$ to $G\T$ for some group $G$, one would need to carefully choose the universe and indexing category $\J$ in order to have a suitably $G$--equivariant result as above. In this paper however, we will only be concerned with the $C_2$--action coming from complex conjugation.
\end{rem}

\section{Polynomial Functors with Reality}\label{section: poly with reality}
Any good theory of calculus, for example \cite{Go90, Go91, Go03, We95, Ta19}, is built on the notion of polynomial functors. These polynomial functors approximate a given functor in such a way to produce a Taylor tower which has strong analogy with Taylor series from differential calculus. The layers of the tower are ``homogeneous functors'' which in each version of functor calculus are characterised by spectra with a particular group action. For example, the orthogonal $n$--homogeneous functors are characterised by spectra with an action of $\O(n)$, \cite[Theorem 7.3]{We95}.

\subsection{The input functors}
We start by describing the category of input functors for calculus with reality. These are the functors one would wish to study in the calculus, and are built using the zeroth jet category. 

\begin{definition}
Define $\CE_0$ to be the category of $C_2\T$--enriched functors from $\CJ_0$ to $C_2\T$.
\end{definition}

This notation is chosen for good reason. One should think of $\CE_0$ as the input category for unitary calculus with an interwoven $C_2$--action. This category comes with several levelwise model structures. The choice of model structures comes from a choice of model structures on $C_2\T$. We choose to work with the Quillen model structure transferred from $\T$ through the adjunction
\[
\adjunction{{C_2}_+ \wedge -}{\T}{C_2\T}{i^*}
\]
where $i$ is the inclusion of the trivial group in $C_2$. The weak equivalences are the underlying weak homotopy equivalences, and the fibrations are the underlying Serre fibrations. The generating cofibrations are of the form $(C_2)_+ \wedge i$ for $i \in I$, and the generating acyclic cofibrations are of the form $(C_2)_+ \wedge j$ for $j \in J$, where $I$ and $J$ denote the set of generating cofibrations and generating acyclic cofibrations for the Quillen model structure on $\T$, respectively. The $C_2$--$CW$--complexes are thus built from cells of the form $(C_2)_+ \wedge D_+^n$. As such we use the following model structure on $\CE_0$.

\begin{prop}
There is a cellular, proper and topological model structure on the category $\CE_0$, where a map $f: E \longrightarrow F$ is a weak equivalence (resp. fibration) if and only if for each $V \in \J^\mathbf{R}$, $f_V : E(V) \longrightarrow F(V)$ is a weak homotopy equivalence (resp. Serre fibration) in $C_2\T$. The generating (acyclic) cofibrations are of the form $\J_0(V,-) \wedge {C_2}_+ \wedge i$ for $i$ a generating (acyclic) cofibration of the Quillen model structure on $\T$.
\end{prop}

One reason for this choice stems from the polynomial model structure. If we were to start with the fixed points model structure on $C_2\T$, we would be unable to verify that the polynomial approximation functors preserves levelwise weak equivalences on fixed points, since this fact relies on the fact that homotopy limits preserves weak equivalences. However homotopy colimits do not, in general, commute with fixed points, and hence the polynomial approximation functors would not interact well with the fixed--points model structure on $\CE_0$. As such, we could not verify the existence of a Bousfield--Friedlander local model structure as in \cite[Proposition 6.5]{BO13} and \cite[Proposition 3.9]{Ta19}, when using the fixed points model structure.


\begin{rem}
The theory of a calculus with reality has previously been studied by Tynan in his thesis \cite{Ty16}. Tynan considered functors from the category of real inner product spaces, but frequently uses the complexification functor to extend to the category of complexified real inner product spaces, which is what we consider here. We feel that our approach is more natural since many of the functors one would wish to consider in a calculus with reality come from the unitary calculus, rather than orthogonal calculus. It should be noted that there is an equivalence of categories between the input category of Tynan and our input category. Moreover, his use of the complexification functor feeds into the authors previous work on comparing orthogonal and unitary calculi \cite{Ta20}.  We further prefer our approach as it make clearer the equivariance involved, and classifies the $n$--homogeneous functors, which is a theorem noticeably absent from \cite{Ty16}. 
\end{rem}

\begin{definition}
Define $\mathbf{R}_{\leq n}$ to be the poset of non--zero subspaces of $\C \otimes_\R \R^n$ with a real basis, and $\mathbf{R}_{<n}$ be the poset of proper non--zero subspaces of $\C \otimes_\R \R^n$ with a real basis.
\end{definition}

\begin{rem}
Note that the poset of all non--zero subspaces of $\C \otimes_\R \R^n (\cong \C^n)$ is closed under complex conjugation, i.e. the line $(1,i)$ is sent to $(1, -i)$ which although is not proportional, is still in the poset. This requirement of real bases comes from the definition of functors in $\CE_0$, i.e. they can only take values on inner product spaces with real bases.
\end{rem}

As a direct corollary of the existence of the projective model structure, we achieve the following.

\begin{cor}\label{domains and codomains of S_n are cofibrant}
The objects $\CJ_n(V,-)$ and $S\Cgamma_{n+1}(V,-)_+$ are cofibrant in $\CE_0$ for each $n \geq 0$.
\end{cor}
\begin{proof}
By the definition of the model structure, $\CJ_0(V,-)$ is cofibrant. The sphere bundle $S\Cgamma_n(V,-)_+$ is homeomorphic to the homotopy colimit
\[
\underset{U \in \mathbf{R}_{< n+1}}{\hocolim}~\CJ_0(U \oplus V, -),
\]
and hence $S\Cgamma_n(V,-)_+$ is cofibrant by \cite[Theorem 18.5.2(1)]{Hi03}.
It follows that since $ S\Cgamma_n(V,-)_+$ and  $\CJ_0(V,-)$ are both cofibrant, and mapping cones of a map between cofibrant objects are cofibrant, that $\CJ_{n+1}(V,-)$ is cofibrant, as the mapping cone of  $S\Cgamma_n(V,-)_+ \longrightarrow \CJ_0(V,-)$.
\end{proof}

\subsection{Polynomial functors with reality} \label{subsection: poly with reality}

Polynomial functors are objects of $\CE_0$ which satisfy extra conditions making them behave like polynomial functions from differential calculus. We start with the definition. This is similar to \cite[Definition 5.1]{We95} and \cite[Definition 2.1]{Ta19}.

\begin{definition}
A functor $F \in \CE_0$ is \textit{polynomial of degree less than or equal $n$} or \textit{$n$--polynomial} if the canonical map
\[
\rho: F(V) \longrightarrow \underset{U \in \mathbf{R}_{\leq n}}{\holim} F(V \oplus U) =: \tau_n F(V)
\]
is a weak equivalence in $C_2\T$.
\end{definition}

\begin{rem}\label{rem: C2 action on holim}
The homotopy limit is constructed to take into account the fact that the poset $\mathbf{R}_{\leq n+1}$ is a category object in the category of spaces, see \cite[Section 4]{We95} and \cite[Appendix]{Lin09}. In particular, the homotopy limit is the totalisation of a cosimplicial space, and hence can be expressed as an enriched end. It  acquires a $C_2$--action from the general fact that  if a diagram has a $G$--action, the end inherits such an action via the following induced map on equalisers
\[
\xymatrix{
\int_i X_{i,i} \ar[r] \ar@{-->}[d]_{g_{\int X}} & \underset{i}{\prod} X_{i,i} \ar@<0.9ex>[r]^s \ar@<-0.9ex>[r]_t \ar[d]_{\prod g_{X_{i,i}}} & \underset{\alpha: i \longrightarrow j}{\prod}X_{j,i} \ar[d]_{\prod g_{X_{j,i}}} \\
\int_i X_{i,i} \ar[r]  & \underset{i}{\prod} X_{i,i} \ar@<0.9ex>[r]^s \ar@<-0.9ex>[r]_t  & \underset{\alpha: i \longrightarrow j}{\prod} X_{j,i}. \\
}
\]
\end{rem}

There is an alternative characterisation of when a functor is $n$--polynomial which is essential for characterising the fibrant objects in the $n$--polynomial model structure. The proof follows as in \cite[Proposition 5.3]{We95} upon noting the $C_2$--equivariance of Proposition \ref{hocolim sphere homeo}.

\begin{prop}\label{relation of sphere bundle and holim}
Let $F \in \CE_0$. Then $F$ is $n$--polynomial if and only if
\begin{equation*}
p^*: F(V) \longrightarrow \CE_0(S\gamma_{n+1}(V,-)_+,F)
\end{equation*}
is a weak homotopy equivalence.
\end{prop}

We now move on to discussing the polynomial approximations of a functor. This definition is completely analogous to those from orthogonal and unitary calculus, see \cite[Definition 6.3]{BO13} or \cite[Definition 3.4]{Ta19}.

\begin{definition}
The \textit{$n$--polynomial approximation}, $T_nF$, of $F \in \CE_0$ is the homotopy colimit of the diagram
\[
\xymatrix{
F \ar[r]^\rho & \tau_nF \ar[r]^\rho & \tau_n^2 F \ar[r]^\rho & \cdots .
}
\]
\end{definition}

\begin{rem}
As an example of a coend, the homotopy colimit inherits a $C_2$--action by the dual construction to Remark \ref{rem: C2 action on holim}.
\end{rem}

The functor $T_nF$ is $n$--polynomial for all $F$. This is a key result, required to prove the existence of the $n$--polynomial model structure. In order to prove this result, we extend the erratum to orthogonal calculus, \cite{We98}, to the calculus with reality setting. To show that $T_nF$ is $n$--polynomial we show that $T_nF \longrightarrow \tau_n T_nF$ is a levelwise weak equivalence. The first step is the following. The proof of which follows from \cite[Lemma e.7, Diagram e.8 and Diagram e.9]{We98}.

\begin{lem}\label{factorising square diagram}
The commutative square
\begin{equation} \label{2}
\xymatrix{
F(V) \ar[r] \ar[d]_{\rho_F} & T_n F(V) \ar[d]^{\rho_{T_nF}} \\
\tau_nF(V) \ar[r] & \tau_n (T_n F)(V)
}
\end{equation}
can be enlarged to a commutative diagram

\begin{equation} \label{3}
\xymatrix{
F(V) \ar[r] \ar[d]_{\rho_F} &  X \ar[d]_g \ar[r] &T_n F(V) \ar[d]^{\rho_{T_nF}} \\
\tau_nF(V) \ar[r] & Y \ar[r] & \tau_n (T_n F)(V)
}
\end{equation}
where $g: X \longrightarrow Y$ is a weak homotopy equivalence.
\end{lem}

Since $T_nF$ is a sequential homotopy colimit, and $\tau_n$ commutes with sequential homotopy colimits (see the proof of \cite[Lemma 5.14]{We95}), $\tau_n T_n F$ is levelwise equivalent to $T_n \tau_n F$. We now prove the required result.

\begin{lem}\label{Tn is n--polynomial}
If $F  \in \CE_0$, then $T_nF$ is $n$--polynomial.
\end{lem}
\begin{proof}
It suffices to show that the vertical arrows in the diagram
\[
\xymatrix{
F(V) \ar[r]^{\rho} \ar[d]_{\rho} & \tau_n F(V) \ar[r]^\rho \ar[d]_\rho & \tau_n^2 F(V) \ar[r]^\rho \ar[d]_\rho & \dots \\
\tau_n F(V) \ar[r]_{\tau_n\rho} & \tau_n^2 F(V) \ar[r]_{\tau_n^2 \rho} & \tau_n^3 F(V) \ar[r]_{\tau_n^3 \rho} & \dots
}
\]
induce a weak homotopy equivalence, $r : T_n F \longrightarrow T_n \tau_n F$, between the homotopy colimits of the rows. For each $k \geq 0$, we have a commutative diagram
\[
\xymatrix{
\tau_n^kF(V) \ar@{^{(}->}[r] \ar[d]_{\rho}& T_nF(V) \ar[d]^{r} \\
\tau_n^{k+1} F(V) \ar@{^{(}->}[r] & T_n \tau_n F(V).
}
\]
Lemma \ref{factorising square diagram} gives a factorisation of this diagram
\begin{equation*}
\xymatrix{
\tau_n^kF(V) \ar[r] \ar[d]_{\rho}&  X \ar[r] \ar[d]  & T_nF(V) \ar[d]^{r} \\
\tau_n^{k+1} F(V) \ar[r] & Y \ar[r] & T_n \tau_n F(V).
}
\end{equation*}
where $X\longrightarrow Y$ is a weak equivalence. Applying homotopy groups yields a diagram of sets, and a diagram chase argument establishes the injectivity and surjectivity of $\pi_*(r)$. It follows that $r$ is a levelwise weak equivalence.
\end{proof}

As with the orthogonal and unitary calculi, polynomial functors satisfy many useful properties. For a full account see \cite{We95, BO13} and \cite{Ta19} for the orthogonal and unitary versions, respectively. Here we give the required properties to construct a suitable $n$--polynomial model structure for functors with reality. The following is the calculus with reality version of \cite[Lemma 5.11]{We95}, the proof of which is similar since homotopy limits commute in $C_2\T$. 

\begin{lem}\label{tau an n--poly}
If $F \in \CE_0$ is $m$--polynomial then $\tau_n E$ is $m$--polynomial for all $n\geq 0$.
\end{lem}

\subsection{The $n$--polynomial model structure}

A key aspect of the work of Barnes and Oman \cite{BO13} is the $n$--polynomial model structure. This model structure captures the homotopy theory of $n$--polynomial functors -- they are the fibrant objects --  and the $n$--polynomial approximation functor is a fibrant replacement in this model structure. Since we are using the underlying model structure on $\CE_0$, producing the $n$--polynomial model structure for calculus with reality and follows from the orthogonal and unitary counterparts, \cite[Proposition 6.6]{BO13} and \cite[Proposition 3.9]{Ta19}.

\begin{prop}\label{prop: n--poly model structure}
There is a cellular, proper and topological model structure on the category $\CE_0$ with the weak equivalences those maps $f: E \longrightarrow F$ such that $T_nf: T_nE \longrightarrow T_nF$ is a weak equivalence in the underlying model structure on $\CE_0$. The fibrations are those maps $f: E \longrightarrow F$, which are levelwise fibrations in $\CE_0$ and the square
\[
\xymatrix{
E \ar[r] \ar[d]_f & T_nE \ar[d]^{T_nf} \\
F \ar[r] & T_nF
}
\]
is a homotopy pullback. The cofibrations of this model structure are the cofibrations of the underlying model structure on $\CE_0$. We call this the $n$--polynomial model structure and denote it by $n\poly\CE_0$.
\end{prop}
\begin{proof}
The polynomial approximation functor satisfies the required conditions of \cite[Theorem 9.3]{Bo01}, and hence the model structure exists, and is proper and topological. Moreover, one can show similarly to \cite[Proposition 6.6]{BO13} and \cite[Proposition 3.9]{Ta19}, that this model structure is the left Bousfield localisation of the underlying model structure on $\CE_0$ at the set of maps
\[
\mathcal{S}_n =\{ S\Cgamma_{n+1}(V,-)_+ \longrightarrow \CJ_0(V,-) \mid V \in \CJ_0\}.
\]
Indeed, since the cofibrations in the $T_n$--local model structure agree with the cofibrations of the left Bousfield localised model structure, it suffices to show that a map is a $T_n$--equivalence if and only if it is a $\mathcal{S}_n$--local equivalence. As the weak equivalences of the projective model structure are the underlying weak equivalences, the proof follows as in \cite[Proposition 6.6]{BO13} and \cite[Proposition 3.9]{Ta19}.
\end{proof}

\section{Homogeneous Functors with Reality}\label{section: homogeneous functors}

As with the orthogonal and unitary calculus, there are maps $T_nF \longrightarrow T_{n-1}F$, which assemble into a Taylor tower under the functor $F \in \CE_0$

\[
\xymatrix@C+1cm{
		&			&	\ar@/_1pc/[dl]	F \ar[d]	  \ar@/^1pc/[drr]  \ar@/^1.3pc/[drrr]   &			     	&			& \\
 \cdots \ar[r] & T_{n+1}F \ar[r]_{r_{n+1}} & T_nF \ar[r]_{r_n} & \cdots \ar[r]_{r_2} & T_1F \ar[r]_{r_1} & T_0F. \\
 & D_{n+1}F \ar[u] & D_nF \ar[u] & & D_1F \ar[u] &\\
}
\]
In order to obtain information about the functor $F$ from this tower, we would like to be able to compute the layers. Indeed, there is a spectral sequence, the \textit{Weiss spectral sequence} associated to $F$ at $V \in \CJ_0$. This is the homotopy spectral sequence of the tower of pointed spaces $\{T_nF\}_{n \in \mathbb{N}_0}$ with $E^1$--page
\[
E^1_{p,q} = \pi_{q-p} D_{p}F(V)
\]
which abuts to $\pi_* \underset{n \in \mathbb{N}_0}{\holim} F(V)$.

The $n$--th layer of the Taylor tower satisfies the property that it is both $n$--polynomial and has trivial $(n-1)$--polynomial approximation. 
Functors with these properties are called $n$--homogeneous.

\begin{definition}
A functor with reality is said to be \textit{$n$--reduced} if it has trivial $(n-1)$--polynomial approximation. A functor $F \in \CE_0$ is \textit{homogeneous of degree less than or equal $n$} or \textit{$n$--homogeneous} if it is $n$--polynomial and its $(n-1)$--polynomial approximation is trivial.
\end{definition}

\subsection{Homogeneous model structure}
A right Bousfield localisation of the $n$--polynomial model structure, as with \cite[Proposition 6.9]{BO13} and \cite[Proposition 4.13]{Ta19} produces a model structure which captures the homotopy theory of functors which are homogeneous of degree less than or equal $n$.

The weak equivalences of this model structure are defined via the derivatives, see Definition \ref{def: derivative}. By \cite[Proposition 8.2]{Ta19}, we could also define the model structure using the $D_n$--equivalences, that is, those maps $f: E \longrightarrow F$ in $\CE_0$ such that $D_nf: D_n E \longrightarrow D_nF$ is  a levelwise weak equivalence. 

\begin{prop}\label{prop: homogeneous model structure with reality}
There is a topological model structure on $\CE_0$ where the weak equivalences are those maps $f$ such that $\ind_0^n
T_nf$ is a weak equivalence in $\CE_0$, the fibrations are the fibrations of the $n$--polynomial model structure and the cofibrations are those maps with the left lifting property with respect to the acyclic fibrations. The fibrant objects are $n$--polynomial and the cofibrant--fibrant objects are the projectively cofibrant $n$--homogeneous functors. We denote this model category by $n\homog\E_0$. 
\end{prop}
\begin{proof}
The $n$--polynomial model structure is cellular and right proper, Proposition \ref{prop: n--poly model structure}, hence by \cite[Theorem 5.1.1.]{Hi03}, we can right Bousfield localise $n\poly\CE_0$ at the set of objects,
\[
\mathcal{K}_n = \{ \CJ_n(V,-) \mid V \in \CJ_0 \}
\]
to achieve the stated model structure. Indeed, a map $f: E \longrightarrow F$ is a $\mathcal{K}_n$--equivalence if and only
\[
(\ind_0^n T_n E)(V) = \CE_0(\CJ_n(V,-), T_nE) \longrightarrow \CE_0(\CJ_n(V,-), T_n F)= (\ind_0^n T_nF)(V)
\]
is a weak homotopy equivalence. It follows that $f$ is a weak equivalence in $R_{K_n}(n\poly\CE_0)$ if and only if $\res_0^n\ind_0^n T_nf$ is a weak equivalence in $\CE_0$.
\end{proof}

In Section \ref{section: homog with reality}, we give a characterisation of homogeneous functors in terms of spectra, similar to those for orthogonal \cite[Theorem 7.3]{We95} and unitary \cite[Theorem 8.1]{Ta19} calculi. This characterisation is both computationally friendly and describes the homotopy theory of $n$--homogeneous functors in terms of spectra. This is advantageous as the category of spectra is well understood.

\section{Derivatives of Functors with Reality}\label{section: derivative}
The other crucial ingredient in a theory of calculus is that of derivatives. Following the constructions of the derivatives in orthogonal and unitary calculus, \cite[Section 2]{We95}, \cite[Section 4]{BO13} \cite[Section 4]{Ta19}, we construct the derivatives of a functor with reality, utilising the Real Stiefel combinatorics developed in Section \ref{section: equivariant stiefel combinatorics}.  
There is a strong analogy between the derivatives of a functor, and the derivatives of a function. For example, both compute the errors between successive polynomial approximations. Moreover, the derivative of any functor may be converted to a spectrum with an appropriate group action. This conversion of the derivative into a spectrum formed one aspect of the zig--zag of Quillen equivalences for orthogonal and unitary calculi, \cite[Proposition 8.3]{BO13}, \cite[Corollary 6.5, Proposition 6.7]{Ta19}. In the `with reality' setting, more work must be done to give a description of the derivative as a spectrum, and forms two steps in the three--step zig--zag of Quillen equivalences.

\subsection{The derivative}
In the orthogonal and unitary calculus the derivative of a functor is constructed as a left Kan extension along a particular inclusion of categories. We follow a similar procedure here and define the derivative in this way. We first define categories, $\CE_n$, for all $n \geq 0$, which should be thought of as the unitary calculus categories with interwoven $C_2$--action coming from complex conjugation.

\begin{definition}
For all $n \geq 0$, define $\CE_n$ to be the category of $C_2\T$--enriched functors from $\CJ_n$ to $C_2\T$.
\end{definition}

For $m< n$, the inclusion $i_m^n : \C^m \longrightarrow \C^n$ onto the first $m$--components is $C_2$--equivariant and induces a $C_2\T$--enriched functor $i_m^n : \CJ_m \longrightarrow \CJ_n$. On the level of vector bundles, the map is given by
\[
i_m^n : \Cgamma_n(U,V) \longrightarrow \Cgamma_{n+1}(U,V), \quad (f,x) \mapsto (f, i_m^n \otimes \id (x)),
\]
which is $C_2$--equivariant.

\begin{definition}\label{def: derivative}
Define the \textit{restriction functor} $\res_m^n : \CE_n \longrightarrow \CE_m$ by precomposition with $i_m^n$. Define the \textit{induction functor} $\ind_m^n : \CE_m \longrightarrow \CE_m$ to be the (enriched) left Kan extension along $i_m^n$. When $m=0$, this induction functor defines the \textit{$n$--th derivative} of a functor with reality.
\end{definition}

The definition of derivative is not amenable to calculations. The following gives the derivative as a particular homotopy fibre, allowing for calculations in some cases.

\begin{prop}\label{fibre sequence for derivative}
For $F \in \CE_n$, there is a homotopy fibre sequence
\[
\res_n^{n+1}\ind_n^{n+1} F(V) \longrightarrow F(V) \longrightarrow \Omega^{2n}F(V \oplus \C)
\]
in $C_2\T$ for all $V \in \J_n$.
\end{prop}
\begin{proof}
This follows from \cite[Lemma 4.6]{BO13} upon noting the $C_2$--equivariance of Proposition \ref{prop: cofibre sequence of jet categories}.
\end{proof}

In both the orthogonal and unitary calculi, relating polynomial functors to derivatives through fibres sequences was incredibly useful for computations of the derivatives and proving the zig--zag of Quillen equivalences. The first step in this process is the following $C_2$--equivariant homotopy cofibre sequence.

\begin{lem}\label{sphere bundle cofibre sequence}
There is a cofibre sequence
\[
S\Cgamma_{n+1}(V,W)_+ \xrightarrow{p_1} \CJ_0(V,W) \longrightarrow \CJ_{n+1}(V,W)
\]
in $C_2\T$ where $p_1$ is the projection onto the first component of $S\gamma_{n+1}(V,W)$.
\end{lem}
\begin{proof}
The mapping cone of $S\gamma_{n+1}(V,W)_+ \xrightarrow{p_1} \CJ_0(V,W)$ is the pushout
\[
\xymatrix{
S\Cgamma_{n+1}(V,W)_+ \ar[r]^{p_1} \ar[d] & \CJ_0(V,W) \ar[d] \\
S\Cgamma_{n+1}(V,W)_+ \wedge [0, \infty] \ar[r] & \mathcal{P}
}
\]
where we use $[0, \infty] = [0, \infty)^c$ (with basepoint $\infty$). This pushout consists of $3$--tuples $(f,x,t)$ for $t \in [0, \infty]$ and $(f,x) \in S\Cgamma_{n+1}(V,W)_+$ modulo the relations
\[
\begin{split}
(f,x,\infty) &= (f',x', \infty) \\
(f,x,0) &= (f',x', 0). \\
\end{split}
\]
The required map $\Phi: \mathcal{P} \longrightarrow \CJ_{n+1}(V,W)$ is given (away from the basepoint) by $\Phi(f,x,t) = (f,xt)$. This map is a $C_2$--equivariant homeomorphism.
\end{proof}

We can now state the desired fibre sequence, which gives a measure of the failure of a functor from being polynomial in terms of the derivative.

\begin{prop}\label{fibre sequence for polynomials}
For $F \in \CE_0$, there is a homotopy fibre sequence
\[
\res_0^{n+1} \ind_0^{n+1}F(V) \longrightarrow F(V) \longrightarrow \tau_n F(V)
\]
in $C_2\T$.
\end{prop}
\begin{proof}
By Lemma \ref{sphere bundle cofibre sequence}, there is a cofibre sequence
\[
S\Cgamma_{n+1}(V,-)_+ \xrightarrow{p_1} \CJ_0(V,-) \longrightarrow \CJ_{n+1}(V,-).
\]
Applying the functor $\CE_0(-,F)$ we obtain a homotopy fibre sequence
\[
\CE_0(\CJ_{n+1}(V,-), F) \longrightarrow \CE_0(\CJ_0(V,-), F) \longrightarrow \CE_0(S\Cgamma_{n+1}(V,-), F)
\]
which reduces via definitions, the Yoneda Lemma and Proposition \ref{relation of sphere bundle and holim} to give a homotopy fibre sequence
\[
\ind_0^n F(V) \longrightarrow F(V) \longrightarrow \tau_nF(V). \qedhere
\]
\end{proof}

As a corollary we see that the $(n+1)$--st derivative of an $n$--polynomial functor is trivial.

\begin{cor}\label{higher derivative vanishes}
If $F \in \CE_0$ is $n$--polynomial, then $\res_0^n\ind_0^n F$ is levelwise weakly contractible.
\end{cor}

A useful result for showing an object is $n$--polynomial is the following result relating $n$--polynomial objects and homotopy fibres. The proof of which is an application of the Five Lemma. In particular the following result gives that the homotopy fibre of a map between $n$--polynomial objects is $n$--polynomial. 

\begin{lem}[{\cite[Lemma 5.5]{We95}}]\label{hofibre lemma for polynomials} \label{hofibre of n--poly is n--poly}
Let $E \in \CE_0$ be $n$--polynomial, $g: E \longrightarrow F$ a morphism in $\CE_0$ and suppose that the $(n+1)$--th derivative of $F$ is trivial Then the functor given by
\begin{equation*}
V \mapsto \hofibre[E(V) \xrightarrow{g_V} F(V)]
\end{equation*}
is an $n$--polynomial.
\end{lem}

An instant corollary is that the functor $V \mapsto \Omega F(V)$ is $n$--polynomial whenever $F^{(n+1)}$ vanishes.

\begin{cor}\label{loops on polynomials}
Let $F \in \CE_0$. If $F^{(n+1)}$ is trivial, then the functor $V \mapsto \Omega F(V)$ is $n$--polynomial.
\end{cor}
\begin{proof}
Apply Lemma \ref{hofibre lemma for polynomials} with $E = \ast$.
\end{proof}

\begin{ex}\label{ex: loop space polynomial}
Let $\Theta \in \s^\mathbf{O}[C_2 \ltimes \U(n)]$. Then the functor given by
\[
V \longmapsto \Omega^\infty[(S^{nV} \wedge \Theta)_{h\U(n)}]
\]
is $n$--polynomial.
\end{ex}
\begin{proof}
This follows from \cite[Example 4.12]{Ta19} upon checking the $C_2$--equivariance.
\end{proof} 

By the Quillen equivalence of the category of spectra with an action of $C_2 \ltimes \U(n)$ and $\CE_1[\U(n)]$, Theorem \ref{thm: E_1 as OS}, we achieve the following corollary. 

\begin{cor}\label{Weiss 5.7 for w/ reality}
Let $\Theta \in \CE_1[\U(n)]$. Then the functor $F$ given by
\[
V \longmapsto \Omega^\infty[(S^{nV} \wedge \Theta)_{h\U(n)}]
\]
is $n$--polynomial.
\end{cor}
\begin{proof}
We sketch the proof, full details of the argument can be found in \cite[Example 5.7]{We95} or \cite[Example 4.12]{Ta19}. Let $\Theta  \in \CE_1[\U(n)]$ and define a functor $F$ as above. The sequence of derivatives of $F$
\[
F^{(n)} \longrightarrow F^{(n-1)} \longrightarrow \dots \longrightarrow F^{(i)} \longrightarrow \dots \longrightarrow F^{(1)} \longrightarrow F.
\]
can be identified with the sequence of functors
\[
 F[n] \longrightarrow F[n-1] \longrightarrow \dots \longrightarrow F[i] \longrightarrow \dots \longrightarrow F[1] \longrightarrow F,
\]
where $F[i](U) = \Omega^\infty[(S^{nU} \wedge \Theta)_{h\U(n-i)}]$, and $\U(n-i)$ fixes the first $i$ coordinates. Note that $F[i]$ is an object of $C_2 \ltimes \U(i)\E_i^\mathbf{R}$ (see \cite[Example 4.12]{Ta19}). The result then follows from Corollary \ref{loops on polynomials}, and by noting that the first derivative of $F[i]$ is $F[i+1]$. 
\end{proof}

\subsection{The intermediate category} \label{subsection: intermediate category}
As with orthogonal calculus, \cite[Section 4]{BO13}, and unitary calculus, \cite[Section 4]{Ta19}, the $n$--th derivative of a functor naturally lands in a category which is intermediate between the input category and spectra with an action of $C_2 \ltimes \U(n)$. Utilising the theory of diagram spaces of Mandell, May, Schwede and Shipley \cite{MMSS01}, we give the construction of such a category here, and its relation with the input category $\CE_0$.

\begin{definition}
Define $C_2 \ltimes \U(n)\E_n^\mathbf{R}$ to be the category of $(C_2 \ltimes \U(n))\T$--enriched functors from $\CJ_n$ to $(C_2 \ltimes \U(n))\T$.
\end{definition}

This category comes with an equivalent description in terms of a category of modules. Let $\CI$ be the category with the same objects as $\CJ$ and linear isometric isomorphisms. Denote by $C_2\CI\T$ and $(C_2\ltimes \U(n))\CI\T$ the categories of $C_2$--equivariant $\CI$--spaces and $(C_2 \ltimes \U(n))$--equivariant $\CI$--spaces, respectively. These are closed symmetric monoidal categories, \cite[Theorem 1.7]{MMSS01}, with product given by Day convolution \cite[Definition 21.4]{MMSS01}.

Define $n\mathbb{S} : \CI \longrightarrow C_2\T$ to be the functor given by $V \mapsto S^{nV}$, where $nV := \C^n \otimes V$, with $C_2$ acting on $S^{nV}$ via complex conjugation on $V$. Following orthogonal and unitary calculus \cite[Lemma 7.3]{BO13}, \cite[Subsection 4.3]{Ta19}, the functor $n\mathbb{S}$ is a commutative monoid in $C_2\CI\T$, and also in $(C_2\ltimes \U(n))\CI\T$. We verify this claim for $C_2\CI\T$, the other case follows from the unitary case \cite[Subsection 4.3]{Ta19}.

\begin{lem}
For each $n \geq 0$, $n\mathbb{S}$ is a commutative monoid in the category $C_2\CI\T$, of $C_2$--equivariant $\CI$--spaces.
\end{lem}
\begin{proof}
The multiplication is identical to that of \cite[Lemma 7.3]{BO13}. It suffices to verify that the evaluation map
\[
\Ev: \CI(V,W)_+ \wedge S^V \longrightarrow S^W
\]
is $C_2$--equivariant. Away from the $C_2$-fixed basepoint, this map is given by $(f, v) \mapsto f(v)$. This is clearly $C_2$--equivariant since
\[
\Ev(g\cdot (f,v)) = \Ev(gfg, gv) = gf(ggv) = g(f(v)) = g\cdot \Ev(f,v). \qedhere
\]
\end{proof}

With this, we get the description of the categories $C_2 \ltimes \U(n)\E_n^\mathbf{R}$ and $\CE_n$ as categories of modules over $n\mathbb{S}$.

\begin{prop}
The category $\CE_n$ is equivalent to the category of $n\mathbb{S}$--modules in $C_2\CI\T$, and the category $C_2 \ltimes \U(n)\E_n^\mathbf{R}$ is equivalent to the category of $n\mathbb{S}$--modules in $(C_2 \ltimes \U(n))\mathbf{R}\T$.
\end{prop}
\begin{proof}
By \cite[Proposition 5.2]{Ta19} it is enough to check that the isomorphism
\[
\int^{U \in \CI} \CI(V \oplus U,W)_+ \wedge S^{nU} \longrightarrow \CJ_n(V,W)
\]
is $C_2$--equivariant. It suffices to check this on the map
\[
\Phi: \CI(V \oplus U,W)_+ \wedge S^{nU} \longrightarrow \CJ_n(V,W), (f,u) \mapsto (f|_V, (\C^n \otimes f)(u))
\]
which construct the above isomorphism. Indeed, let $g \in C_2$ and recall $g=g^{-1}$. Then
\[
\begin{split}
g \cdot \Phi(f,u) &= g \cdot (f|_V, (\C^n \otimes f_(u))) = (g\cdot f|_V, g((\C^n \otimes f)(u)))  \\
			    &= ((gfg)|_V, g((\C^n \otimes f(ggu)))) = \Phi( g\cdot f, gu) = \Phi(g\cdot(f,u)). \qedhere
\end{split} 
\]
\end{proof}

We now follow the procedure set by Barnes and Oman \cite[Section 4]{BO13} for orthogonal calculus and combine the restriction--induction adjunction with change of group functors of Mandell and May \cite[Section V.2]{MM02} to construct an adjunction between $C_2 \ltimes \U(n)\E_n^\mathbf{R}$ and $C_2 \ltimes \U(m)\E_m^\mathbf{R}$ similar to that of the above adjunction between $\CE_n$ and $\CE_m$. Note that the same technique was employed in \cite[Section 7]{Ta19}.

\begin{definition}
Define the \textit{restriction--orbit functor} $\res_m^n/\U(n-m) : C_2 \ltimes \U(n)\E_n^\mathbf{R} \longrightarrow C_2 \ltimes \U(m)\E_m^\mathbf{R}$ as the functor which sends $X$ to $(X \circ i_m^n)/\U(n-m)$.
\end{definition}

This is as well defined functor since $(X \circ i_m^n)/\U(n-m)$ is a $(C_2 \ltimes \U(m))\T$--enriched functor from $\CJ_m$ to $(C_2 \ltimes \U(m))\T$.

The restriction-orbit functor has a right adjoint, the first step in the construction of which, is to identify the right adjoint of the orbits functor $(-)/\U(n-m) : (C_2 \ltimes \U(n))\T \longrightarrow (C_2 \ltimes \U(m))\T$. This right adjoint is defined as the composite of two functors. The first takes a $(C_2 \ltimes \U(m))$--space $Y$ and considers it as a $(C_2 \ltimes (\U(m) \times \U(n-m)))$--space by letting the $\U(n-m)$--factor act trivially. We denote this by $\varepsilon^*Y$. The second functor takes $\varepsilon^*Y$ and sends to the space of $(C_2 \ltimes (\U(m) \times \U(n-m)))$--equivariant maps from $C_2 \ltimes \U(n)$ to $\varepsilon^*Y$. The result is as adjoint pair
\[
\adjunction{(-)/\U(n-m)}{(C_2 \ltimes \U(n))\T}{(C_2 \ltimes \U(m))\T}{CI_m^n},
\]
where $CI_m^nY = F_{C_2 \ltimes (\U(m) \times \U(n-m))} ((C_2 \ltimes \U(n))_+, \varepsilon^*Y)$.

The result is an adjunction
\[
\adjunction{\res_m^n/\U(n-m)}{C_2 \ltimes \U(n)\E_n^\mathbf{R}}{C_2 \ltimes \U(m)\E_m^\mathbf{R}}{\ind_m^nCI},
\]
where $(\ind_m^nCI)(X)(V) = C_2 \ltimes \U(m)\E_m^\mathbf{R}(\CJ_n(V,-), CI_m^nX)$.

We are particularly interested in the case when $m=0$ in which instance, the adjunction reduces to
\[
\adjunction{\res_0^n/\U(n)}{C_2\ltimes \U(n)\E_n^\mathbf{R}}{\CE_0}{\ind_0^n\varepsilon^*}.
\]
in which case $\ind_0^n \varepsilon^*F$ is the $n$-th derivative of $F$. 

\subsection{The $n$--stable model structure} With the description of $C_2 \ltimes \U(n)\E_n^\mathbf{R}$ as $n\mathbb{S}$--modules we can construct a stable model structure on $C_2 \ltimes \U(n)\E_n^\mathbf{R}$. We start as is standard with the projective model structure, and suitably left Bousfield localise to produce the $n$--stable model structure. The existence of this model structure follows from \cite{MMSS01}, taking $\J_n$ as the diagram category.

\begin{lem}
There is a cellular, proper and topological model structure on $C_2 \ltimes \U(n)\E_n^\mathbf{R}$ with weak equivalences and fibrations defined levelwise. The generating cofibrations are of the form
\[
\CJ_n(U,-) \wedge (C_2 \ltimes \U(n))_+ \wedge i
\]
for $i$ a generating cofibration of $\T$. The generating acyclic cofibrations are of the form
\[
\CJ_n(U,-) \wedge (C_2 \ltimes \U(n))_+ \wedge j
\]
for $j$ a generating acyclic cofibration of $\T$.
\end{lem}

A left Bousfield localisation of this model structure results in the $n$--stable model structure. The construction is completely analogous to the $n$--stable model structure of \cite[Section 7]{BO13} or \cite[Subsection 4.3]{Ta19}.

\begin{definition}
The \textit{$n$--homotopy groups} of $X \in C_2 \ltimes \U(n)\E_n^\mathbf{R}$ are defined as
\[
n\pi_k X = \underset{q}{\colim}~ \pi_{2nq +k} X(\C^q).
\]
A map $f: X \longrightarrow Y$ in $C_2 \ltimes \U(n)\E_n^\mathbf{R}$ is an \textit{$n\pi_*$--isomorphism} if $n\pi_k : n\pi_kX \longrightarrow n\pi_kY$ is an isomorphism for all $k\geq 0$.
\end{definition}

There is also a version of $n\Omega$--spectra defined analogously to the orthogonal or unitary calculus.

\begin{definition}
An element $X \in C_2 \ltimes \U(n)\E_n^\mathbf{R}$ is an \textit{$n\Omega$--spectrum} if the adjoint structure maps
\[
X(U) \longrightarrow \Omega^{nV}X(U \oplus V)
\]
are weak homotopy equivalences for all $U, V \in \CJ_n$.
\end{definition}

Denote by
\[
\lambda_{V,W}^n : \CJ_n(V \oplus W, -) \wedge S^{nW} \longrightarrow \CJ_n(V,-)
\]
the restricted composition map. We can factor this map, through it's mapping cylinder, as a cofibration
\[
k_{V,W}^n : \CJ_n(V \oplus W, -) \wedge S^{nW}  \longrightarrow M\lambda_{V,W}^n
\]
and an acyclic fibration
\[
r_{V,W}^n : M\lambda_{V,W}^n \longrightarrow \CJ_n(V,-).
\]

Adding the cofibrations $\{k_{V,W}^n\}$ to the acyclic cofibrations of the projective model structure in a particular way yields the $n$--stable model structure.

\begin{thm}
There is a cofibrantly generated, proper, and topological model structure on the category $C_2 \ltimes \U(n)\E_n^\mathbf{R}$, where the weak equivalences are the $n\pi_*$--isomorphisms, and the fibrations are those maps $f: X \longrightarrow Y$ which are levelwise fibrations, such that the square
\[
\xymatrix{
X(V) \ar[r] \ar[d]_{f_V} & \Omega^{nW}X(V \oplus W) \ar[d]^{\Omega^{nW}f_{V\oplus W}} \\
Y(V) \ar[r] & \Omega^{nW}Y(V \oplus W)
}
\]
is a homotopy pullback for all $V,W \in \CJ_n$. The generating cofibrations are those of the projective models structure and the generating acyclic cofibrations are the union of the projective generating acyclic cofibrations together with the set
\[
K_{V,W}^n \Box I := \{ k_{V,W}^n \Box i \ : \ i \in I, \ V,W \in \CI\},
\]
where $k_{V,W}^n \Box i$ denotes the pushout product of the maps $k_{V,W}^n$ and $i$.
\end{thm}

The derivatives of $n$--polynomial objects are well behaved with respect to the $n$--stable model structure, in that they are $n\Omega$--spectra. The orthogonal version of this may be found in \cite[Proposition 5.12]{BO13} or \cite[Proposition 5.12]{We95}, the proof is completely analogous. 

\begin{lem}\label{derivatives of n--poly are stable}
If $E$ is an $n$--polynomial in $\CE_0$, then for any $V \in \CJ_0$, the map
\[
\ind_0^nE(V) \longrightarrow \Omega^{2n} \ind_0^nE(V \oplus \C),
\]
adjoint to the structure maps of $\ind_0^nE$ is a weak homotopy equivalence.
\end{lem}

\section{The Intermediate Category as a Category of Spectra}\label{section: intermediate category as spectra}
In orthogonal calculus, Weiss \cite{We95}, constructs a zig--zag of equivalences up to homotopy between $n$--homogeneous functors and spectra with an action of $\O(n)$. Barnes and Oman \cite{BO13} improved this result to a zig--zag of Quillen equivalences between orthogonal spectra with an action of $\O(n)$, $\os[\O(n)]$, and the $n$--homogeneous model structure on the category of orthogonal functors, $n\homog\E_0^\mathbf{O}$. This result was extended to unitary calculus in \cite{Ta19}.

This result also holds in the setting of calculus with reality, albeit, it is slightly more complicated as the equivariance requires the introduction of a further step in the zig--zag of Quillen equivalences. We start by showing that the intermediate category $C_2 \ltimes \U(n)\E_n^\mathbf{R}$ is equivalent to the category of $\U(n)$--objects in $\CE_1$, $\CE_1[\U(n)]$.

We further give a description of the category $\CE_1$ in terms of the Real spectra (Definition \ref{def: real spectra}) as defined by Schwede, \cite[Example 7.7]{Sch19}. Moreover, we give a description of Real spectra (or $\CE_1$) in terms of orthogonal spectra with an action of the group $C_2 \ltimes \U(n)$. This section results in the following zig--zag Quillen equivalences between the intermediate category a spectra with a group action.
\[
\xymatrix@C+1cm{
C_2 \ltimes \U(n)\E_n^\mathbf{R}
\ar@<+1ex>[r]^{(\xi_n)_!}
&
\CE_1[\U(n)]
\ar@<+1ex>[l]^{(\xi_n)^*}
\ar@<-1ex>[r]_{\psi}
&
\s^\mathbf{O}[C_2 \ltimes \U(n)]
\ar@<-1ex>[l]_{L_\psi} \\
}
\]
The rest of this section is dedicated to explaining the above equivalences.

\begin{rem}\label{algebraic model remark}These Quillen equivalences describe the $n$-th derivatives in terms of spectra with a $(C_2 \ltimes \U(n))$--action. There are two main methods to convert a spectrum with an action of $C_2 \ltimes \U(n)$ into a genuine $(C_2 \ltimes \U(n))$--spectrum. The first is to note that our stable model structure on $\s[C_2 \ltimes \U(n)]$ is a model for free $G$--spectra since it is homotopically compactly generated by the suspension spectrum of $C_2 \ltimes \U(n)$, see Greenlees and Shipley \cite[Section 3]{GS11} for a discussion on other models. The other option is that of cofree $(C_2 \ltimes \U(n))$--spectra. In \cite[Lemma 5.3]{Ke17}, K\k{e}dziorek demonstrates a Quillen equivalence between spectra with a $G$--action and cofree $G$-spectra, modelled by the $EG_+$--localisation of the category of genuine $G$--spectra. Another method of describing a spectrum with a $G$-action as a cofree $G$-spectrum is given by Hill and Meier in \cite[Subsection 2.2]{HM17}, as the derived functor $\mathrm{I} F(EG_+, -)$, where $\mathrm{I}$ is the equivalence of categories between spectra with a $G$-action and genuine $G$-spectra induced by in the inclusion of a trivial $G$-representations into a complete $G$-universe.

The benefit of these descriptions of spectra with a $G$-action as genuine $G$-spectra are the algebraic models for their rational homotopy type. Greenlees and Shipley \cite{GS11, GS14}, provide an algebraic model for rational free $G$--spectra, through a Quillen equivalence, \cite[Theorem 1.1]{GS14}, between rational free $G$--spectra and torsion $H^*\widetilde{BN}[W]$--modules, where $N$ is the identity component of $G$ and $W=G/N$. In our particular case, rational free $(C_2 \ltimes \U(n))$--spectra is algebraically modelled by torsion $H^*\widetilde{B\U(n)}[C_2]$-modules. Pol and Williamson \cite{PW19}, further provide an algebraic model for rational cofree $G$--spectra in the form of a Quillen equivalence, \cite[Theorem 9.6]{PW19}, between rational cofree $G$--spectra and $L_0^I$--complete differential--graded $H^*\widetilde{BN}[W]$--modules, where $I$ is the augmentation ideal of $H^*BN$, and $L_0^I$ is the zeroth left derived functor of $I$--adic completion. In our case, the Pol and Williamson algebraic model is given by the category of $L_0^I$--complete differential graded $H^*\widetilde{B\U(n)}[C_2]$--modules, where $I$ is the augmentation ideal of the polynomial ring $H^*B\U(n) = \Q[c_1, \cdots, c_n]$ on the first $n$ Chern classes, that is, the ideal generated by the Chern classes, $I=(c_1, \cdots, c_n)$. As such, there are two (Quillen equivalent) algebraic models for the rational homotopy type of $n$-homogeneous functors.
\end{rem}

\subsection{The Quillen equivalence between $C_2 \ltimes \U(n)\E_n^\mathbf{R}$ and $\CE_1[\U(n)]$}

An object $X$ of $\CE_1[\U(n)]$ is given by a collection of spaces $\{X(V) \mid V \in \J^\mathbf{R}_1\}$ with a $(C_2 \ltimes \U(n))$--action with structure maps
\[
S^2 \wedge X(V) \longrightarrow X(V \oplus \C).
\]
There structure maps are $(C_2 \ltimes \U(n))$--equivariant with diagonal action on the domain, trivial $\U(n)$--action on $S^2$, and $C_2$ acting on $S^2$ by complex conjugation, since $C_2$ acts on  $\U(n)$ via complex conjugation.

Following \cite[Section 8]{BO13} and \cite[Subsection 4.3]{Ta19}, we construct the adjunction via a functor on the indexing categories. Define $\xi_n : \CJ_n \longrightarrow \CJ_1$ by $\xi_n(V) = \C^n \otimes_\C V$ on objects, and $\xi_n(f,x) = (\C^n \otimes_\C f, x)$ on morphisms. This induces a functor
\[
(\xi_n)^* : \CE_1[\U(n)] \longrightarrow C_2 \ltimes \U(n)\E_n^\mathbf{R}
\]
given by precomposition, where we let $C_2 \ltimes \U(n)$ act on $X(nV)$  by $X(g\sigma \otimes V) \circ (g\sigma)_{X(nV)}$ where, $g \in C_2$, and $\sigma \in \U(n)$. Here, $X(g\sigma \otimes V)$ is the internal action on $X(nV)$ induced by the action on $nV$, and $(g\sigma)_{X(nV)}$ is the external action on $X(nV)$ induced by $X(nV)$ being a $(C_2 \ltimes \U(n))$--space. Checking that this functor is well defined is equivalent to checking that the map
\[
(\xi_n)^*X : \CJ_n(U,V) \longrightarrow \T((\xi_n)^*X(U), (\xi_n)^*X(V))
\]
is $(C_2 \ltimes \U(n))$--equivariant.

To see this, consider the following commutative diagram
\[
\xymatrix@C+1cm{
\CJ_n(U,V) \ar[r]^{\xi_n} \ar[d]^{(g\sigma)} & \CJ_1(nU, nV) \ar[r]^X \ar[d]^{((g\sigma)^{-1} \otimes U)^*\circ((g\sigma) \otimes V)_*} & \T(X(nU), X(nV)) \ar[d]^{(X((g\sigma)^{-1} \otimes U))^*\circ(X((g\sigma) \otimes V))_*} \\
\CJ_n(U,V) \ar[r]_{\xi_n}& \CJ_1(nU,nV) \ar[r]_X & \T(X(nU), X(nV)). \\
}
\]
Let $(f,x) \in \CJ_n(U,V)$, applying $X \circ \xi_n$ to $(f,x)$ gives a $(C_2 \ltimes \U(n))$--equivariant map $X(nU) \longrightarrow X(nV)$, and it follows that the two expressions
\[
\begin{split}
X((g\sigma) \otimes V) \circ &X(\C^n \otimes f, x) \circ X((g\sigma)^{-1} \otimes U), \\
(g\sigma)_{X(U)} \circ X(\sigma \otimes V) \circ &X(\C^n \otimes f, x) \circ X((g\sigma)^{-1} \otimes U) \circ (g\sigma)^{-1}_{X(V)}
\end{split}
\]
are equal. Note removing the $C_2$--action gives the exact proof of this fact for the unitary calculus.

Left Kan extending along $\xi_n$ defines the left adjoint $(\xi_n)_!$ to precomposition with $\xi_n$, resulting in an adjunction
\[
\adjunction{(\xi_n)_!}{C_2 \ltimes \U(n)\E_n^\mathbf{R}}{\CE_1[\U(n)]}{(\xi_n)^*}.
\]
This left adjoint comes with the usual description as a $(C_2 \ltimes \U(n))\T$--enriched coend,
\[
(\xi_n)_!(X)(V) = \int^{U \in \CJ_n} \CJ_1(nU, V) \wedge X(U).
\]

\begin{thm}\label{thm: intermediate cat as E_1}
The adjoint pair
\[
\adjunction{(\xi_n)_!}{C_2 \ltimes \U(n)\E_n^\mathbf{R}}{\CE_1[\U(n)]}{(\xi_n)^*}
\]
is a Quillen equivalence.
\end{thm}
\begin{proof}
The proof follows similarly to \cite[Theorem 6.8]{Ta19}. It is straightforward to show that the right adjoint preserves acyclic fibrations and fibrant objects. Moreover, a confinality argument demonstrates that the right adjoint reflects weak equivalences.

It is left to show that the derived unit of the adjunction is an equivalence. Since $C_2 \ltimes \U(n)\E_n^\mathbf{R}$ is homotopically compactly generated it suffices to show the derived unit condition on the homotopically compact generator $(C_2 \ltimes \U(n))_+ \wedge n\mathbb{S}$ of $C_2 \ltimes \U(n)\E_n^\mathbf{R}$. This is a matter of plugging the homotopically compact generator into the formula for the unit, as in \cite[Theorem 6.8]{Ta19}.
\end{proof}

\subsection{The equivalence between $\CE_1$ and Real spectra} The above Quillen equivalence provides an equivalence of categories between the homotopy category of $C_2 \ltimes \U(n)\E_n^\mathbf{R}$ and the homotopy category of $\CE_1[\U(n)]$. In both orthogonal and unitary calculus, the categories $\E_1^\mathbf{O}$ and $\E_1^\mathbf{U}$ are equivalent to the categories of orthogonal and unitary spectra respectively. We now show that $\CE_1$ is equivalent to the category of Real spectra of Schwede, see \cite[Example 7.11]{Sch19}. 

\begin{definition}\label{def: real spectra}
A Real spectrum $X$ is a sequence of spaces $\{X_k\}_{k \in \mathbb{N}}$ with an action of $C_2 \ltimes \U(k)$ together with structure maps
\[
\sigma_k : X_k \wedge S^2 \longrightarrow X_{k+1}
\]
such that the iterated structure maps
\[
\sigma_k^m : X_k \wedge S^{2m} \longrightarrow X_{k+m}
\]
are $(C_2 \ltimes (\U(k) \times \U(m)))$--equivariant.
\end{definition}

A map of Real spectra $f: X \longrightarrow Y$ is then a collection of maps $f_k : X_k \longrightarrow Y_k$ which are compatible with the structure maps in the usual sense. Hence there is a category of Real spectra, denoted $C_2 \ltimes \us$. Again this notation is deliberate as one should think of Real spectra as unitary spectra with an interwoven $C_2$--action.

\begin{prop} \label{prop: real spectra and E1}
The category of Real spectra, $C_2 \ltimes \us$ is equivalent to the category $\CE_1$.
\end{prop}
\begin{proof}
We construct an inverse equivalence of categories. Define $\mathbb{U} : \CE_1 \longrightarrow C_2 \ltimes \us$ by
\[
(\mathbb{U}X)_n = X(\C^n).
\]
The space $(\mathbb{U}X)_n$ inherits a $(C_2 \ltimes \U(n))$--action from the evaluation maps
\[
\J^\mathbf{R}(V,W)_+ \wedge X(V) \longrightarrow X(W)
\]
under the special case $V = W = \C^n$. The iterated structure maps
\[
\sigma^m : (\mathbb{U}X)_n \wedge S^m \longrightarrow (\mathbb{U}X)_{n+m}
\]
are induced by the structure maps for $X$, and are appropriately $(C_2 \ltimes (\U(n)\times \U(m))$--equivariant by the special case $V=V'=\C^n$ and $W=W'=\C^m$ of the commuting of the diagram
\[
\xymatrix{
\J^\mathbf{R}(V, V')_+ \wedge \J^\mathbf{R}(W,W')_+ \wedge X(V) \wedge S^W \ar[r] \ar[d] & \J^\mathbf{R}(V \oplus W, V' \oplus W')_+ \wedge X(V \oplus W) \ar[d] \\
X(V') \wedge S^{W'} \ar[r] & X(V' \oplus W').
}
\]
In the other direction, define $\mathbb{P}: C_2 \ltimes \us \longrightarrow \CE_1$, by
\[
(\mathbb{P}Y)(V) = \J^\mathbf{R}(\C^n, V)_+ \wedge_{\U(n)} Y_n
\]
whenever $\dim(V)=n$. $\U(n)$ acts on $\J^\mathbf{R}(\C^n , V)$ by precomposition, $C_2$--acts diagonally on the smash product, and $(\mathbb{P}Y)(V)$ is then the coequaliser of the two $\U(n)$--actions on the smash product. Any choice of isometry $\varphi: \C^n \longrightarrow V$ defines a homeomorphism
\[
[\varphi, -] : Y_n \longrightarrow (\mathbb{P}Y)(V), x \mapsto [\varphi, x].
\]
The $C_2$--action is then given by $g [\varphi x] = [g\varphi, gx]$. Moreover, the iterated structure maps 
\[
\sigma^m : Y_n \wedge S^m \longrightarrow X_{n+m}
\]
are a special case of the generalised structure maps
\[
\sigma_{V,W} : (\mathbb{P}Y)(V) \wedge S^W \longrightarrow (\mathbb{P}Y)(V \oplus W)
\]
which are defined by setting $m= \dim(W)$ and choosing an isometry $\psi: \C^m \longrightarrow W$. Then
\[
\sigma_{V,W}([\varphi, x],w) = [\varphi \oplus \psi, \sigma^m(x \wedge \psi^{-1}(w))].
\]
By construction, the homeomorphism induced by $\varphi = \id$ demonstrates that $\mathbb{U}\mathbb{P} \cong \mathbbm{1}$, and $\mathbb{P}\mathbb{U} \cong \mathbbm{1}$.
\end{proof}

\subsection{The Quillen equivalence between $\CE_1[\U(n)]$ and spectra with an action of $C_2 \ltimes \U(n)$.}
Using the work of Schwede \cite[Example 7.11]{Sch19}, we give a Quillen equivalence between $\CE_1$ and the category of $C_2$--objects in orthogonal spectra, $\s^\mathbf{O}[C_2]$. 

Let $X \in \CE_1$, and define a functor $\psi : \CE_1 \longrightarrow \os[C_2]$ by
\[
\psi(X)(V) = \Omega^{iV} X(\C \otimes V) = \T(S^{iV}, X(\C \otimes V)).
\]

Assuming $\dim V=n$, then the group $C_2 \times \O(n)$ acts on $iV$ by the sign representation for the $C_2$--factor and via the regular representation for the $\O(n)$--factor. Moreover $C_2 \times \O(n)$ acts on $X(\C \otimes V)$ via restriction along the inclusion $C_2 \times \O(n) \hookrightarrow C_2 \ltimes \U(n)$, and hence $C_2 \times \O(n)$ acts on the mapping space via conjugation.

The description as a mapping space gives a clear description of the structure maps,
\[
\begin{split}
S^1 \wedge \T(S^{iV}, X(\C \otimes V)) &\xrightarrow{\text{assemble}} \T(S^{iV} , S^1 \wedge  X(\C \otimes V)) \\
									  &\xrightarrow{S^{i\R} \wedge -} \T(S^{i\R} \wedge S^{iV}, S^{iR} \wedge S^1 \wedge X(\C \otimes V)) \\
									  &\xrightarrow{\cong} \T(S^{i(\R \oplus V)}, S^2 \wedge X(\C \otimes V)) \\
									  &\xrightarrow{(\sigma_{\C \otimes V})_*} \T(S^{i(\R \oplus V)},  X(\C \otimes (V \oplus \R))) \\
									  &\xrightarrow{=} \psi(X)(V \oplus \R)\\
\end{split}
\]
where we use the $C_2$--equivariant decomposition $\R \oplus i\R \cong \C$ to identity $S^{\R \oplus i\R}$ with $S^\C$. The functor $\psi$ has a left adjoint, giving an adjunction between $\CE_1$ and $\s^\mathbf{O}[C_2]$.

\begin{prop}\label{prop: psi adjoint pair}
There is an adjoint pair
\[
\adjunction{L_{\psi}}{\os[C_2]}{\CE_1}{\psi}
\]
where
\[
L_\psi(X)(\C \otimes V) = \int^{U \in \J_1^\mathbf{O}} \CJ_1(\C \otimes U, \C \otimes V) \wedge X(U) \wedge S^{iU}.
\]
\end{prop}
\begin{proof} A standard calculus of (co)ends argument verifies the claim. 
\end{proof}

This adjunction produces a Quillen equivalence between $\CE_1$ and $\s^\mathbf{O}[C_2]$.

\begin{prop}\label{prop: QE between orth spec and CE_1}
The adjoint pair
\[
\adjunction{L_{\psi}}{\os[C_2]}{\CE_1}{\psi}
\]
is a Quillen equivalence when both categories are equipped with their stable model structures.
\end{prop}
\begin{proof}
By \cite[Proposition 8.5.4 and Lemma 7.7.1]{Hi03}, in order to to exhibit a Quillen adjunction, it suffices to show that the right adjoint preserves acyclic fibrations and fibrant objects. Let $f: X \longrightarrow Y$ be an acyclic fibration in $\CE_1$. Then, $f:X \longrightarrow Y$ is a levelwise acyclic fibration. By construction $\psi(f) : \psi X \longrightarrow \psi Y$ will also be a levelwise acyclic fibration, and hence an acyclic fibration in $\os[C_2]$.

Now, let $X$ be a fibrant object in $\CE_1$, then $X(V) \simeq \Omega^{\C}X(V \oplus \C)$ for all $V \in \CJ_1$. It follows that
\[
\begin{split}
\psi(X)(U) = \Omega^{iU}X(\C \otimes U) &\simeq \Omega^{iU} \Omega^{\C} X((\C \otimes U) \oplus \C) \\
									    &\simeq \Omega^{iU} \Omega^{\R \oplus i\R} X(\C \otimes (U \oplus \R))\\
									    &\simeq \Omega^\R \Omega^{i(U \oplus \R)} X(\C \otimes (U \oplus \R)) \\
									    &\simeq \Omega^\R \psi(U \oplus \R),\\
\end{split}
\]
and hence, the right adjoint preserves fibrant objects.

For the Quillen equivalence, it is left to show that the right adjoint reflects weak equivalences and that the derived unit is an isomorphism. For the first, suppose that $f: X \longrightarrow Y$ is a map in $\CE_1$ such that $\psi(f): \psi X \longrightarrow \psi Y$ is a $\pi_*$--isomorphism of spectra. Then
\[
\begin{split}
\pi_k (\psi X) &= \underset{q}{\colim}~ \pi_{k+q} (\psi X(\R^q)) = \underset{q}{\colim}~ \pi_{k+q}~ (\Omega^{i\R^q} X(\C \otimes \R^q)) \cong \underset{q}{\colim}~ \pi_{k+2q}~ (\psi X(\C^q))= \pi_k (X).
\end{split}
\]
A similar argument yields $\pi_* (\psi Y) \cong \pi_* Y$, and hence $\psi$ reflects weak equivalences.

To show that the derived unit is an isomorphism, it suffices, by \cite[Lemma 3.2.]{Ke17}, to show that the derived unit on the homotopically compact generator $\Sigma^\infty_+ C_2$ of $\os[C_2]$. Plugging the homotopically compact generator $\Sigma^\infty_+ C_2$ into the definition of the left adjoint as a coend produces a levelwise equivalence with $(C_2)_+ \wedge \J_1^\mathbf{R}(0,-)$, the compact generator of $\CE_1$, which after applying the right adjoint is stably equivalent to $\Sigma^\infty_+ C_2$.
\end{proof}

This Quillen equivalence extends to a Quillen equivalence between $\CE_1[\U(n)]$ and $\os[C_2 \ltimes \U(n)]$. The right adjoint $\psi$ may be constructed $(C_2 \ltimes \U(n))$--equivariantly and hence the left adjoint may be constructed to take this $(C_2 \ltimes \U(n))\T$--enrichment into account.

\begin{lem}
There is an adjoint pair
\[
\adjunction{L_\psi}{\os[C_2 \ltimes \U(n)]}{\CE_1[\U(n)]}{\psi}
\]
where $\psi : \CE_1[\U(n)] \longrightarrow \os[C_2 \ltimes \U(n)]$ is given by
\[
\psi(X)(V) =  \T(S^{iV}, X(\C \otimes V)).
\]
\end{lem}
\begin{proof}
The functor $\psi$ is well defined, as it's structure maps of $\psi X$, are suitably $(C_2 \ltimes \U(n))$--equivariant. To see this, notice that the structure maps of $\psi X$ are defined using precomposition with the structure maps of $X$, which are known to be $C_2 \ltimes \U(n))$--equivariant. The construction of the left adjoint and the proof of the adjunction then follows similarly to Proposition \ref{prop: psi adjoint pair}.
\end{proof}

The Quillen equivalence from Proposition \ref{prop: QE between orth spec and CE_1} extends to a Quillen equivalence of these categories since all homotopical considerations are given by the underlying (non--equivariant) homotopy theory.

\begin{thm}\label{thm: E_1 as OS}
The adjoint pair
\[
\adjunction{L_{\psi}}{\os[C_2 \ltimes \U(n)]}{\CE_1[\U(n)]}{\psi}
\]
is a Quillen equivalence when both categories are equipped with their stable model structures.
\end{thm}
\begin{proof}
The Quillen adjunction follows from Proposition \ref{prop: QE between orth spec and CE_1}, as does the fact that the right adjoint reflects weak equivalences. It is left to show that the unit is a derived isomorphism. Indeed, the left adjoint applied to the homotopically compact generator $\Sigma^\infty_+(C_2 \ltimes \U(n))$ is isomorphic $(C_2 \ltimes \U(n))_+ \wedge \CJ_1(0,-)$ in $\CE_1[\U(n)]$. Hence, it suffices (\cite[Lemma 3.2]{Ke17}) to check that the derived unit is an isomorphism on the homotopically compact generator. The unit map
\[
\Sigma^\infty_+(C_2 \ltimes \U(n))(V) \xrightarrow{\qquad \eta \qquad} \Omega^{iV} \left( \int^{U \in \J_0^\mathbf{O}} \J_1^{\mathbf{R}}(\C \otimes U, \C \otimes V) \wedge S^U \wedge (C_2 \ltimes \U(n))_+ \wedge S^{iU} \right)
\]
is induced by the unit of the $(\Sigma, \Omega)$-adjunction and the map into the coend for the case $V=U$. There is a commutative diagram 
\[
\xymatrix{
\Sigma^\infty_+(C_2 \ltimes \U(n)) \ar[r]^\eta \ar[dr] & \psi L_\psi(\Sigma^\infty_+(C_2 \ltimes \U(n)) \ar[d] \\
& \psi(\J_1^\mathbf{R}(0,-) \wedge(C_2 \ltimes \U(n))_+)
}
\]
where the vertical map is induced by the isomorphism $L_\psi(\Sigma^\infty_+(C_2 \ltimes \U(n)))(V) \cong \J_1^\mathbf{R}(0,V) \wedge (C_2 \ltimes \U(n))_+ $, and the diagonal map is a stable equivalence. It follows that the unit of the adjunction is also a stable equivalence. 
\end{proof}

\section{Differentiation as a Quillen Functor}\label{section: diff as QF}
With the model structures for calculus with reality in place, we can show that the differentiation functor, $\ind_0^n \varepsilon^*$, is a right Quillen functor as part of a Quillen equivalence between the $n$--homogeneous model structure on the category of functors with reality, $n\homog\CE_0$, and the $n$--stable model structure on the intermediate category $C_2 \ltimes \U(n)\E_n^\mathbf{R}$. The process of constructing such a Quillen equivalence is similar to \cite[Section 9]{BO13} and \cite[Section 7]{Ta19}. This Quillen equivalence will further allow for the classification of $n$--homogeneous functors in terms of orthogonal spectra with an action of $C_2 \ltimes \U(n)$, again, similarly to \cite[Theorem 7.3]{We95}, \cite[Theorem 8.1]{Ta19}.

We start by proving a Quillen adjunction between the underlying projective model structures on $\CE_0$ and $C_2 \ltimes \U(n)\E_n^\mathbf{R}$, then extend this, via the interplay between Bousfield localisations and Quillen adjunctions of Hirschhorn \cite[Chapter 3]{Hi03}, to a Quillen adjunction between the $n$--homogeneous model structure, $n\homog\CE_0$, and the $n$--stable model structure on $C_2 \ltimes \U(n)\E_n^\mathbf{R}$.

\begin{lem}\label{lem: QA between projective model structures}
For $n\geq 0$, there is a Quillen adjunction
\[
\adjunction{\res_0^n/\U(n)}{C_2 \ltimes \U(n)\E_n^\mathbf{R}}{\CE_0}{\ind_0^n \varepsilon^*}
\]
where both categories are equipped with the projective model structure.
\end{lem}
\begin{proof}
Both model structures are cofibrantly generated, hence by \cite[Lemma 2.1.20]{Ho99} it suffices to show that the left adjoint preserves the generating (acyclic) cofibrations.

The generating (acyclic) cofibrations of the projective model structure on $C_2 \ltimes \U(n)\E_n^\mathbf{R}$ are of the form
\[
(C_2 \ltimes \U(n))_+ \wedge \J_n(U,-)  \wedge i
\]
where $i$ is a generating (acyclic) cofibration of the projective model structure on $\T$. The result follows since $\res_0^n \J_n(U,-)$ is cofibrant in $\CE_0$, by Corollary \ref{domains and codomains of S_n are cofibrant}.
\end{proof}

Using the composition of Quillen adjunctions, we achieve the following extension of Lemma \ref{lem: QA between projective model structures} to the $n$--polynomial model structure.

\begin{lem}\label{lem: QA with projective and n--poly models}
For $n\geq 0$, there is a Quillen adjunction
\[
\adjunction{\res_0^n/\U(n)}{C_2 \ltimes \U(n)\E_n^\mathbf{R}}{n\poly\CE_0}{\ind_0^n \varepsilon^*},
\]
where $C_2 \ltimes \U(n)\E_n^\mathbf{R}$ is equipped with the projective model structure.
\end{lem}
\begin{proof}
The $n$--polynomial model structure is a left Bousfield localisation of the underlying model structure on $\CE_0$, hence by \cite[Proposition 3.3.4]{Hi03}, there is a Quillen adjunction
\[
\adjunction{\mathbbm{1}}{\CE_0}{n\poly\CE_0}{\mathbbm{1}}.
\]
Composition of this Quillen adjunction with the Quillen adjunction of Lemma \ref{lem: QA between projective model structures} results in the desired Quillen adjunction by \cite[Subsection 1.3.1]{Ho99}.
\end{proof}

Localisation theorems of Hirschhorn \cite[Theorem 3.1.6 and Proposition 3.3.18]{Hi03}, give criteria for when Quillen adjunctions may be extended to left or right Bousfield localisations. As such, Lemma \ref{lem: QA with projective and n--poly models} may be extended to the $n$--stable model structure.

\begin{lem}\label{lem: QA between n--stable and n--poly}
For $n\geq 0$, there is a Quillen adjunction
\[
\adjunction{\res_0^n/\U(n)}{C_2 \ltimes \U(n)\E_n^\mathbf{R}}{n\poly\CE_0}{\ind_0^n \varepsilon^*}
\]
where $C_2 \ltimes \U(n)\E_n^\mathbf{R}$ is equipped with the $n$--stable model structure.
\end{lem}
\begin{proof}
By \cite[Theorem 3.1.6, Proposition 3.3.18]{Hi03}, it suffices to show that the right adjoint sends fibrant objects in the n--polynomial model structure to $n\Omega$--spectra, i.e. that the derivative of an $n$--polynomial functor is an $n\Omega$--spectrum. This is precisely the content of Lemma \ref{derivatives of n--poly are stable}.
\end{proof}

The $n$--homogeneous model structure was constructed as a right Bousfield localisation of the $n$--polynomial model structure in Proposition \ref{prop: homogeneous model structure with reality}. As such, we can extend Lemma \ref{lem: QA between n--stable and n--poly}, again using \cite[Proposition 3.3.18]{Hi03}, to the $n$--homogeneous model structure.

\begin{lem}\label{lem: QA between n--stable and n--homog}
For $n\geq 0$, there is a Quillen adjunction
\[
\adjunction{\res_0^n/\U(n)}{C_2 \ltimes \U(n)\E_n^\mathbf{R}}{n\homog\CE_0}{\ind_0^n \varepsilon^*}
\]
where $C_2 \ltimes \U(n)\E_n^\mathbf{R}$ is equipped with the $n$--stable model structure.
\end{lem}
\begin{proof}
Let $f: E \longrightarrow F$ be a $\mathcal{K}_n$--cellular equivalence, between $n$--polynomial objects. By definition of a $\mathcal{K}_n$--cellular equivalence, the map
\[
\CE_0(\J_n(U,-), E) \longrightarrow \CE_0(\J_n(U,-), F)
\]
is a weak homotopy equivalence, and hence by definition $\ind_0^n\varepsilon^* f$ is a levelwise, and hence $n$--stable equivalence. An application of \cite[Proposition 3.3.18]{Hi03} yields the result.
\end{proof}

We have produced a Quillen adjunction between the $n$--stable model structure and the $n$--homogeneous model structure. We now turn our attention to upgrading this Quillen adjunction to a Quillen equivalence. There are several slightly different approaches to this task in the literature, one provided by Barnes and Oman \cite{BO13}, in their study of orthogonal calculus, and the other provided by the author \cite{Ta19} in their study of unitary calculus. We choose to give a slight variation on both these approaches here. We start with a lemma, which is similar to \cite[Lemma 9.3]{BO13}.

\begin{lem}\label{lem: left derived functor characterisation}
The left derived functor
\[
\mathbbm{L}\res_0^n/\U(n) : C_2 \ltimes \U(n)\E_n^\mathbf{R} \longrightarrow \CE_0
\]
is levelwise weakly equivalent to $E\U(n)_+ \wedge_{\U(n)} \res_0^n(-)$.
\end{lem}
\begin{proof}
Let $X \in C_2 \ltimes \U(n)\E_n^\mathbf{R}$, and denote by $\cofrep X$ the projective cofibrant replacement of $X$ in $C_2 \ltimes \U(n)\E_n^\mathbf{R}$. Since $\cofrep X$ is cofibrant in $C_2 \ltimes \U(n)\E_n^\mathbf{R}$, it is in particular levelwise $\U(n)$--free, hence there is a levelwise weak equivalence
\[
E\U(n)_+ \wedge_{\U(n)} \res_0^n (\cofrep X) \longrightarrow E\U(n)_+ \wedge_{\U(n)} \res_0^n X
\]
induced by the levelwise weak equivalence $\cofrep X \longrightarrow X$. The weak homotopy equivalence $E\U(n)_+ \longrightarrow S^0$, induced a levelwise weak equivalence
\[
E\U(n)_+ \wedge_{\U(n)} \res_0^n (X) \longrightarrow \res_0^n X/\U(n),
\]
and the result follows.
\end{proof}

The following is a version of \cite[Example 6.4]{We95}, for calculus with reality. The proof follows similarly as $C_2 \ltimes \U(n)$ is a compact Lie group, and \cite[Example 6.4]{We95} works for general compact Lie groups. 

\begin{ex}\label{ex: loop spaces T_n--equiv}
Let $\Theta \in \s^\mathbf{O}[C_2 \ltimes \U(n)]$. The functors defined by
\[
V \longmapsto \Omega^\infty[(S^{nV} \wedge \Theta)_{h\U(n)}]
\]
and
\[
V \longmapsto [\Omega^\infty(S^{nV} \wedge \Theta)]_{h\U(n)}
\]
are $T_n$--equivalent.
\end{ex}

By the Quillen equivalence of the category of spectra with an action of $C_2 \ltimes \U(n)$ and $\CE_1[\U(n)]$, Theorem \ref{thm: E_1 as OS}, we achieve the following corollary. 

\begin{cor}\label{Weiss 6.4 for w/ reality}
Let $\Theta \in \CE_1[\U(n)]$. The functors defined by
\[
V \longmapsto \Omega^\infty[(S^{nV} \wedge \Theta)_{h\U(n)}]
\]
and
\[
V \longmapsto [\Omega^\infty(S^{nV} \wedge \Theta)]_{h\U(n)}
\]
are $T_n$--equivalent.
\end{cor}

We are now in a position to prove the desired Quillen equivalence. For this, we utilise the Quillen equivalence between the $n$--stable model structure on the intermediate category, and the stable model structure on $\CE_1[\U(n)]$, see Theorem \ref{thm: intermediate cat as E_1}. The reader should compare this proof to \cite[Theorem 10.1]{BO13} and \cite[Theorem 7.5]{Ta19}, as the technique is similar. In \cite{Ta19} we could compose two Quillen equivalences and hence, worked with orthogonal $\U(n)$--spectra, however we only needed to go one step along the zig--zag, and the proof there would have worked just as well with unitary $\U(n)$--spectra. The inability to compose Quillen equivalences in our current situation requires us to work with $\CE_1[\U(n)]$ rather than $(C_2 \ltimes \U(n))$--spectra.

\begin{thm}\label{thm: QE of E_0 and E_n}
For $n\geq 0$, the Quillen adjunction
\[
\adjunction{\res_0^n/\U(n)}{C_2 \ltimes \U(n)\E_n^\mathbf{R}}{n\homog\CE_0}{\ind_0^n \varepsilon^*}
\]
where $C_2 \ltimes \U(n)\E_n^\mathbf{R}$ is equipped with the $n$--stable model structure, is a Quillen equivalence.
\end{thm}
\begin{proof}
The proof follows similarly to \cite[Theorem 10.1]{BO13} and \cite[Theorem 7.5]{Ta19}. We highlight one method for showing that the derived unit is an isomorphism. Given cofibrant $X \in C_2 \ltimes \U(n)\E_n^\mathbf{R}$, there is a commutative diagram
\[
\xymatrix{
X \ar[r] \ar[d] & \cofrep(\xi_n)^*\fibrep(\xi_n)_!X \ar[d] \\
\ind_0^n\varepsilon^* T_n \res_0^n X/\U(n) \ar[r] & \ind_0^n\varepsilon^* T_n \res_0^n (\cofrep (\xi_n)^*\fibrep(\xi_n)_!X)/\U(n) \\
}
\]
where $\cofrep$ denotes cofibrant replacement in $C_2 \ltimes \U(n)\E_n^\mathbf{R}$, and $\fibrep$ denotes fibrant replacement in $\CE_1[\U(n)]$. 

The top horizontal map is a stable equivalence by Theorem \ref{thm: intermediate cat as E_1} and Theorem \ref{thm: E_1 as OS}. The bottom horizontal map is also a weak equivalence as derived functors preserve equivalences. We want to show that the left vertical map is a weak equivalence, as such, it suffices to show that the right vertical map is an equivalence. 

By Lemma \ref{lem: left derived functor characterisation} we can rewrite (up to levelwise equivalence) the codomain of the right hand vertical map as
\[
\ind_0^n\varepsilon^* T_n (E\U(n)_+ \wedge_{\U(n)} \res_0^n ((\xi_n)^*\fibrep(\xi_n)_!X)).
\]
For any object $Y$ of $\CE_1[\U(n)]$ there is a weak equivalence $(\xi_n)^*Y \simeq F(Y)$ where $F(Y)$ is the functor given by 
\[
V \mapsto \underset{k}{\hocolim}~\Omega^{2nk}[S^{nV} \wedge Y].
\]
As such, Corollary \ref{Weiss 6.4 for w/ reality} yields a weak equivalence
\[
T_n (E\U(n)_+ \wedge_{\U(n)} \res_0^n ((\xi_n)^*\fibrep(\xi_n)_!X)) \simeq T_n(F(\fibrep(\xi_n)_!X))
\]
By Corollary \ref{Weiss 5.7 for w/ reality}, $F(\fibrep(\xi_n)_!X)$ is $n$--polynomial, and hence weakly equivalent to its $n$--polynomial approximation. The result then follows by Corollary \ref{Weiss 5.7 for w/ reality}, which gives the $m$--th derivative of $F(\fibrep(\xi_n)_!X)$ as
\[
V \mapsto \underset{k}{\hocolim}~\Omega^{2nk}[E\U(n-m)_+ \wedge_{\U(n-m)} \wedge (S^{nV} \wedge \fibrep(\xi_n)_!X) ]
\]
Taking $m=n$ yields the result. 
\end{proof}

The end result is the following zig-zag of Quillen equivalences relating the $n$-homogeneous model structure on the category of functors with reality and spectra with a $(C_2\ltimes \U(n))$--action.
\[
\xymatrix@C+1cm{
n \homog \CE_0
\ar@<-1 ex>[rr]_(0.6){\ind_0^n \varepsilon^*}
&&
C_2 \ltimes \U(n)\E_n^\mathbf{R}
\ar@<-1ex>[ll]_(0.4){\res_0^n/\U(n)}
\ar[]!<2ex,1ex>;[ddll]!<-3ex,1ex>_{(\xi_n)_!}
 \\
 &&
 \\
\CE_1[\U(n)]
\ar[]!<-2ex,-1ex>;[uurr]!<3ex,-1ex>_{(\xi_n)^*}
\ar@<-1ex>[rr]_{\psi}
&&
\s^\mathbf{O}[C_2 \ltimes \U(n)].
\ar@<-1ex>[ll]_{L_\psi}
}
\]

\section{Classification of Homogeneous Functors with Reality} \label{section: homog with reality}

With the above Quillen equivalence between the $n$--homogeneous model structure on $\CE_0$ and the $n$--stable model structure on $C_2 \ltimes \U(n)\E_n^\mathbf{R}$, we can now give the characterisation of the $n$--homogeneous functors with reality, similar to the characterisation of $n$--homogeneous functors from orthogonal and unitary calculus, see \cite[Theorem 7.3]{We95} and \cite[Theorem 8.1]{Ta19}. Denote by $\Theta_F^n$ the spectrum given by the derived image of $F \in n\homog\CE_0$ under the zig-zag of Quillen equivalences. Under the assumption that $F$, $n$-homogeneous the equivalences of homotopy categories given an equivalence between $\Theta_F^n$ and a spectrum $\Psi_F^n$, with $\Psi_F^n(\R^{2n} \otimes_\R U) = \ind_0^n\varepsilon^*F(\C \otimes U)$. For $\Psi_F^n$ to be a well defined spectrum, it is enough to specify that the iterative structure maps
\[
S^{2n} \wedge \Psi(\R^{2n} \otimes U) \to \Psi_F^n(\R^{2n} \otimes (V \oplus \R))
\]  
are given by the structure maps of $\ind_0^n\varepsilon^* F \in \C_2 \times \U(n)\E_n^\mathbf{R}$. The suspension coordinate does not have trivial $C_2 \ltimes \U(n)$ action, but following the procedure of \cite[Section 3]{We95} we may replace (up to stable equivalence) $\Psi_F^n$ by a spectrum with the correct equivariance. These stable equivalences of spectra are levelwise weak equivalences as all the spectra are $\Omega$-spectra. The following proof is similar to \cite[Theorem 7.3]{We95}, but aided in the use of model categories. The same proof technique was also employed by the author in \cite[Theorem 8.1]{Ta19}.

\begin{thm}\label{thm: homog char}
If $F \in \CE_0$ is $n$--homogeneous, then $F$ is levelwise weakly equivalent to the functor
\[
V \mapsto \Omega^\infty[(S^{\C^n \otimes V} \wedge \Theta_F^n)_{h\U(n)}]
\]
where $\Theta_F^n \in \s^\mathbf{O}[C_2 \ltimes \U(n)]$ is the derived image of $F$ under the zig-zag of Quillen equivalences. 
\end{thm}
\begin{proof}
Let $F$ be cofibrant-fibrant in $n\homog\CE_0$, that is, $F$ is $n$-homogeneous and projectively cofibrant. Define two new objects of $\CE_0$,
\[
\begin{split}
E : \CJ_0 \longrightarrow \T, \ \C \otimes V &\longmapsto (\ind_0^n F(\C \otimes V))_{h\U(n)} \\
G: \CJ_0 \longrightarrow \T, \ \C \otimes V &\longmapsto \Omega^\infty [ (S^{\R^{2n} \otimes V} \wedge \Theta_F^n)_{h\U(n)}]
\end{split}
\]
Since $E(\C \otimes V) \simeq \Theta_F^n(\R^{2n} \otimes V)$, there is a levelwise weak equivalence between $E$  and the object of $\CE_0$ defined by 
\[
\C \otimes V \longmapsto [\Omega^\infty(S^{\R^{2n} \otimes V} \wedge \Theta_F^n)]_{h\U(n)}.
\]
It follows by Example \ref{ex: loop spaces T_n--equiv}, and the fact that $G$ is $n$-polynomial (Example \ref{ex: loop space polynomial}) that there is a zig-zag of levelwise weak equivalence $T_nE \longrightarrow T_nG \longleftarrow G$. As derived functors preserve equivalences, there is a zig-zag of levelwise weak equivalences,
\[
\ind_0^n \varepsilon^* T_nE \longrightarrow \ind_0^n \varepsilon^*T_nG \longleftarrow \ind_0^n\varepsilon^* G.
\]
The $n$-th derivative of $G$ is identified in Example \ref{ex: loop space polynomial} with the functor $G[n] \in C_2  \ltimes \U(n)\E_n^\mathbf{R}$ defined by 
\[
\C \otimes V \longmapsto \Omega^\infty(S^{\R^{2n} \otimes V} \wedge \Theta_F^n).
\]
Since $G[n]$ is levelwise weakly equivalent to $\ind_0^n \varepsilon^*F$, there is a zig-zag of levelwise weak equivalences between $\ind_0^n \varepsilon^* T_nE$ and $\ind_0^n \varepsilon^*F$. A double application of Whitehead's Theorem for (co)localisations of model structures, \cite[Theorem 3.2.13]{Hi03}, yields a zig-zag of levelwise weak equivalences between $E$ and $F$. The case for general $n$-homogeneous $F$ then follows by projectively cofibrantly replacing $F$. 
\end{proof}

\begin{rem}
In the above proof we used the identification $\C^n \otimes_\C \C \otimes_\R V \cong \R^{2n} \otimes_\R V$ to identify their one-point compactifications. Such an identification between the one-point compactifications made the relationship to orthogonal spectra clearer. 
\end{rem}

The end result of this chapter is that given a functor with reality, $F \in \CE_0$, there exists a Taylor tower approximating $F$ at $V\in \CJ$
\[
\xymatrix{
		&	&	F(V) \ar@/_1pc/[dl]	  \ar@/^1pc/[dr]  \ar@/^1.3pc/[drr]   			     	&			& \\
 \cdots \ar[r]_{r_{n+1}} & T_nF(V) \ar[r]_{r_n} & \cdots \ar[r]_{r_2} & T_1F(V) \ar[r]_{r_1} & F(\C^\infty) \\
 &  \Omega^\infty[(S^{nV} \wedge \Psi_F^n)_{h\U(n)}] \ar[u] & & \Omega^\infty[(S^V \wedge \Psi_F^1)_{h\U(1)}] \ar[u] &\\
}
\]
where the $n$--th layer is characterised by an orthogonal spectrum with an action of $C_2 \ltimes \U(n)$.

\section{Examples}\label{section: examples}

The general theory of such a calculus will be familiar to those with some knowledge of the orthogonal or unitary calculus. In this section we gather a series of examples for the reader, so those familiar with the general idea of the theory can see this new calculus in practice, and refer back to the relevant sections where necessary. There is a stark similarity between calculus with reality and unitary calculus. This should not be surprise as one may thing of unitary calculus as the resulting calculus after `forgetting' the $C_2$--action which we have built into the calculus with reality. 

\subsection{Representable functors with reality}
We are particularly interested in the representable functors. They played a crucial role in understanding convergence results in orthogonal and unitary calculus \cite{Ta19}, and interact well with the comparison functors of \cite{Ta20}. We now use the model categories developed in this paper to describe the derivatives of the representable functors with reality. Consider the functor $\CJ_n(0,-) = n\mathbb{S}$. We will then extend this to $\CJ_n(V,-)$ for all $V \in \CJ_n$ and all $n\geq 0$.

\begin{ex}\label{ex: nth derivative of nS}
Let $n\mathbb{S} \in n\homog\CE_0$. This is the image of $\U(n)_+ \wedge \CJ_n(0,-)$ under the derived left adjoint $\bbm{L} \res_0^n/\U(n)$. In turn, applying $\bbm{L}(\xi_n)_!$ to $\U(n)_+ \wedge \CJ_n(0,-)$ gives $\U(n)_+ \wedge \J_1^\mathbf{R}(0,-)$. Moreover, $\U(n)_+ \wedge \J_1^\mathbf{R}(0,-)$ is the image of $\Sigma_+^\infty \U(n)$ under the derived left adjoint $\bbm{L} (L_\psi)$. Diagrammatically, we have the following
\[
n\mathbb{S} \longmapsfrom{\U(n)_+ \wedge \CJ_n(0,-)} \longmapsto \U(n)_+ \wedge \CJ_1(0,-) \longmapsfrom \Sigma_+^\infty \U(n).
\]
It follows that the $n$--th derivative of $n\mathbb{S}$ is the (naive) $(C_2 \ltimes \U(n))$--spectrum $\Sigma_+^\infty \U(n)$. %
\end{ex}

\begin{rem}
In the above example we saw that the derived left adjoint, $\bbm{L}(\xi_n)_!$ `changes' $\CJ_n$ to $\CJ_1$. Intuitively, one should think of this functor as a `change of rings' functor. 
\end{rem}

\begin{rem}
As an example of how the algebraic model of Greenlees and Shipley from Remark \ref{algebraic model remark} reduces the complexity of computations, we see by \cite[Corollary 9.2]{GS14} that the algebraic model for the $n$-th derivative of $n\mathbb{S}$ is $\Sigma^{\dim(\U(n))}\Q =\Sigma^{n^2}\Q$. It would be interesting to explore the existence of an algebraic model for the calculus as a whole, in which the $n$-th derivative of the algebraic model for $n\mathbb{S}$ would be $\Sigma^{n^2}\Q$.
\end{rem}

Calculating the $n$--th derivative in Example \ref{ex: nth derivative of nS} allows us to calculate the $n$--polynomial approximation. The unitary version of the following is \cite[Example 9.7]{Ta19}.

\begin{ex}\label{ex: nS n--reduced}
The functor $n\mathbb{S}$ is $n$--reduced. Since $n\mathbb{S}$ is cofibrant in $\CE_0$ and an object of the localising set $\mathcal{K}_n$, the general theory of left and right localisations tells us that $n\mathbb{S}$ is cofibrant in the $n$--homogeneous model structure, and hence $n$--reduced by \cite[Corollary 8.6]{Ta19}. Alternatively, one could note that the map $n\mathbb{S}(U) \longrightarrow \ast$ is $(2n\dim(U) -1)$--connected. The `with reality' version of \cite[Lemma e.7]{We98}, yields a levelwise weak equivalence $T_k (n\mathbb{S}) \longrightarrow T_k(\ast) \simeq \ast$, for all $k \geq n$. In particular, this yields a levelwise weak equivalence 
\[
\begin{split}
T_n (n\mathbb{S})(V) &\simeq \Omega^\infty [(S^{nV} \wedge \Theta_{n\mathbb{S}}^n)_{h\U(n)}] \simeq \Omega^\infty [(S^{nV} \wedge \Sigma_+^\infty \U(n))_{h\U(n)}] \simeq \Omega^\infty\Sigma^\infty [(S^{nV} \wedge \U(n)_+)_{h\U(n)}]. \\
\end{split}
\]
\end{ex}

The above example was the case $\CJ_n(0,-)$. We now examine the general case. For this case we are careful and write the objects of $\J^\mathbf{R}$ as a tensor. 

\begin{ex}
As before, we have the following diagram, where each arrow is a derived left adjoint, as part of the zig--zag of Quillen equivalences. 
\[
\CJ_n(\C \otimes U,-) \longmapsfrom{\U(n)_+ \wedge \CJ_n(\C \otimes U,-)} \longmapsto \U(n)_+ \wedge \CJ_1(\C \otimes U,-) \longmapsfrom \Sigma^\infty_U(\U(n)_+),
\]
where $\Sigma^\infty_U(\U(n)_+)$ is the shift--desuspension of $\U(n)_+$, left adjoint to evaluation at $U$. To see the last arrow, we calculate $\mathbbm{L}(L_\psi)(\Sigma^\infty_U(\U(n)_+)$. Indeed,
\[
\begin{split}
(L_\psi)(\Sigma^\infty_U(\U(n)_+)(\C \otimes V) &= \int^{W \in \J_1^\mathbf{O}} \J_1^\mathbf{R}(\C \otimes W, \C \otimes V) \wedge \Sigma^\infty_U(\U(n)_+)(W) \wedge S^{iW} \\
									&= \int^{W \in \J_1^\mathbf{O}} \J_1^\mathbf{R}(\C \otimes W, \C \otimes V) \wedge \J_1^\mathbf{O}(U,W) \wedge \U(n)_+ \wedge S^{iW} \\
									&\cong \int^{W \in \J_1^\mathbf{O}} \J_1^\mathbf{O}(W, r(\C \otimes V)) \wedge \J_1^\mathbf{O}(U,W) \wedge \U(n)_+ \wedge S^{iW} \\
									&\cong \int^{W \in \J_1^\mathbf{O}} \J_1^\mathbf{O}(U, r(\C \otimes V))  \wedge \U(n)_+ \wedge S^{iW} \\
									&\simeq \J_1^\mathbf{R}(\C \otimes U, \C \otimes V)  \wedge \U(n)_+ \\				
\end{split}
\]
In particular, we see that the derivative of $\CJ_n(U,-)$ is a shift--desuspension of the derivative of $\CJ_n(0,-)$.
\end{ex}

\subsection{The Borel construction on the ($n$--fold) one--point compactification functor with reality}
An interesting functor along the lines of the $n$--fold one--point compactification functor, $n\mathbb{S}$ is the functor given by 
\[
V \mapsto (S^{nV})_{h\U(n)},
\]
that is, the Borel construction on the $n$--fold one--point compactification functor.

\begin{ex}
The $n$--th derivative of the functor $n\mathbb{S}_{h\U(n)}: V \mapsto (S^{nV})_{h\U(n)}$, is the sphere spectrum $\Sigma^\infty S^0$ in $\s^\mathbf{O}[C_2 \ltimes \U(n)]$. As in diagrammatic displays above, we get a diagram
\[
n\mathbb{S}_{h\U(n)} \longmapsfrom \CJ_n(0,-) \longmapsto \CJ_1(0,-) \longmapsfrom \Sigma^\infty S^0.
\]
\end{ex}

We can also calculate the $n$--polynomial approximation of the functor $n\mathbb{S}_{h\U(n)}$.

\begin{ex}
Homotopy orbits do not decrease connectivity, hence by Example \ref{ex: nS n--reduced}, $n\mathbb{S}_{h\U(n)}$ is $n$--reduced. As such, the $n$--polynomial approximation of $n\mathbb{S}_{h\U(n)}$ is given by 
\[
\begin{split}
T_n (n\mathbb{S}_{h\U(n)})(V) &\simeq \Omega^\infty [(S^{nV} \wedge \Theta_{n\mathbb{S}_{h\U(n)}}^n)_{h\U(n)}] \simeq \Omega^\infty [(S^{nV} \wedge \Sigma^\infty S^0)_{h\U(n)}] \simeq \Omega^\infty\Sigma^\infty [(S^{nV})_{h\U(n)}]. \\
\end{split}
\]
\end{ex}

\subsection{The Real classifying space of the unitary group}
Proposition \ref{fibre sequence for derivative} gives a homotopy fibre sequence which allows for the iterative calculation of the derivative of a functor with reality. We can apply this to the Real classifying space of the unitary group functor, $\BU_\R(-) : V \mapsto \BU_\R(V)$, which is given by $\BU(V)$ with $C_2$--action inherited from the complex conjugation on $V$.  In this case, we only calculate the first derivative, to give the reader a feel for the theory. 

\begin{ex}
There is a homotopy fibre sequence
\[
\BU_\R^{(1)}(V) \longrightarrow \BU_\R(V) \longrightarrow \BU_\R(V \oplus \C),
\]
that is, a $C_2$--equivariant homotopy fibre sequence
\[
\BU^{(1)}(V) \longrightarrow \BU(V) \longrightarrow \BU(V \oplus \C),
\]
where $\BU(V)$ had the induced $C_2$--action by complex conjugation on $3V$. As such, the first derivative of $\BU_\R(-)$ is the shifted sphere spectrum $\Sigma^\infty S^{-1} \simeq \Omega \Sigma^\infty S^0$, with $C_2 \ltimes \U(1)$--acting via the $C_2$--action on the inner product spaces, and $\U(1)$ acting trivially.
\end{ex}

We now consider a functor of which, $\BU_\R(-)$ is an extension. The following functor is easier to understand since it has trivial $0$-polynomial approximation. This is \cite[Example 10.2]{We95}, with added $C_2$--equivariance. 

\begin{ex}
Consider the functor $E$ given by 
\[
V \longmapsto \bigslant{\U(V \oplus \C^\infty)}{\U(V)}.
\]
This functor is similar to $\BU_\R(-)$, as $\BU_\R(-)$ is an extension of $E$ by a functor of polynomial degree zero, that is, $T_0 E(V) = E(\C^\infty) \ast$, where $T_0\BU_\R(V) = \BU_\R$. The contractibility of $T_0E$ means we can attempt to calculate $T_1E$ using the Taylor tower. The classification of homogeneous functors, together with the fibre sequence 
\[
D_1 E(V) \longrightarrow T_1 E(V) \longrightarrow T_0E(V) \simeq \ast
\]
yields a levelwise weak equivalence
\[
T_1 E(V) \simeq \Omega^\infty[(S^V \wedge  \Sigma^\infty S^{-1})_{h\U(1)}]
\]
where we have identified the first derivative of $E$, with loops on the orthogonal sphere spectrum $\Omega\Sigma^\infty S^0$. We see that
\[
T_1 E(V) \simeq \Omega Q[(S^V)_{h\U(1)}]
\]
where $Q$ is the stabilisation functor. The $C_2$--action follows through all of these weak homotopy equivalences. 
\end{ex}

\bibliography{references}
\bibliographystyle{plain}
\end{document}